\documentclass[11pt]{article}

\usepackage{epsfig,epsf,fancybox}
\usepackage{amsmath}
\usepackage{mathrsfs}
\usepackage{amssymb}
\usepackage{graphicx}
\usepackage{color}
\usepackage{paralist}
\usepackage{multirow}

\newtheorem{theorem}{Theorem}[section]

\newtheorem{lemma}[theorem]{Lemma}

\newtheorem{definition}{Definition}[section]

\newtheorem{remark}[theorem]{Remark}
\newtheorem{assumption}[theorem]{Assumption}
\numberwithin{equation}{section}

\newcommand{\st}{\textnormal{s.t.}}

\,
\,

\newcommand{\RR}{\mathbf R}
\newcommand{\argmin}{\mathop{\rm argmin}}
\newcommand{\dist}{\mathop{\rm dist}}
\newcommand{\LCal}{\mathcal{L}}

\newcommand{\be}{\begin{equation}}
\newcommand{\ee}{\end{equation}}
\newcommand{\ba}{\begin{array}}
\newcommand{\ea}{\end{array}}
\newcommand{\bpm}{\begin{pmatrix}}
\newcommand{\epm}{\end{pmatrix}}

\newcommand{\XCal}{\mathcal{X}}

\newcommand{\etal}{{et al. }}

\begin{document}

\title{Iteration Complexity Analysis of Multi-Block ADMM for a Family of
Convex Minimization without Strong Convexity}

\author{Tianyi Lin\footnotemark[1]  \and Shiqian Ma\footnotemark[1] \and Shuzhong Zhang\footnotemark[2]}
\renewcommand{\thefootnote}{\fnsymbol{footnote}}
\footnotetext[1]{Department of Systems Engineering and Engineering Management, The Chinese University of Hong Kong, Shatin, New Territories, Hong Kong, China.}
\footnotetext[2]{Department of Industrial and Systems Engineering, University of Minnesota, Minneapolis, MN 55455, USA.}

\date{May 6, 2015}

\maketitle

\begin{abstract}

The alternating direction method of multipliers (ADMM) is widely used in solving structured convex optimization problems due to its superior practical performance.
On the theoretical side however, a counterexample was shown in \cite{Chen-admm-failure-2013} indicating that the multi-block ADMM for minimizing the sum of $N$ $(N\geq 3)$ convex functions with $N$ block variables linked by linear constraints may diverge. It is therefore of great interest to investigate further sufficient conditions on the input side which can guarantee convergence for the multi-block ADMM.
The existing results typically
require the strong convexity on parts of the objective.
In this paper, we present convergence and convergence rate results for the multi-block ADMM applied to solve certain $N$-block $(N\geq 3)$ convex minimization problems {\it without requiring strong convexity}.
Specifically, we prove the following two results: (1) the multi-block ADMM returns an $\epsilon$-optimal solution within $O(1/\epsilon^2)$ iterations by solving an associated perturbation to the original problem; (2) the multi-block ADMM returns an $\epsilon$-optimal solution within $O(1/\epsilon)$ iterations when it is applied to solve a certain {\it sharing problem}, under the condition that the augmented Lagrangian function satisfies the Kurdyka-{\L}ojasiewicz property, which essentially covers
most convex optimization models except for some pathological cases.

\vspace{0.8cm}


%\vspace{0.5cm}

%\noindent {\bf Mathematics Subject Classification 2010:}

\end{abstract}

%\newpage

\section{Introduction}
We consider %solving
the following multi-block convex minimization problem:
\be\label{prob:N}\ba{ll} \min & f_1(x_1) + f_2(x_2) + \cdots + f_N(x_N) \\
                         \st  & A_1 x_1  + A_2 x_2 + \cdots + A_N x_N = b \\
                              & x_i \in \mathcal{X}_i, \, i = 1,\ldots, N, \ea \ee
where $A_i \in \RR^{p\times n_i}$, $b\in\RR^p$, $\XCal_i\subset \RR^{n_i}$ are closed convex sets, and $f_i:\RR^{n_i}\rightarrow\RR$ are closed convex functions.
One effective way to solve \eqref{prob:N}, whenever applicable,  %, when the functions $f_i$'s are of special structures,
is the so-called Alternating Direction Method of Multipliers (ADMM).
The ADMM is closely related to the Douglas-Rachford \cite{Douglas-Rachford-56} and Peaceman-Rachford \cite{Peaceman-Rachford-55} operator splitting methods that date back to 1950s. These operator splitting methods were further studied later in \cite{Lions-Mercier-79,Fortin-Glowinski-1983,Glowinski-LeTallec-89,Eckstein-thesis-89}.
The ADMM has been revisited recently due to its success in solving problems with special structures arising from compressed sensing, machine learning, image processing, and so on; see the recent survey papers \cite{Boyd-etal-ADM-survey-2011,Eckstein-tutorial-admm} for more information.

The ADMM %for solving \eqref{prob:N}
is constructed under an augmented Lagrangian framework, where the augmented Lagrangian function for
\eqref{prob:N} is defined as
\[ \LCal_\gamma(x_1,\ldots,x_N;\lambda) := \sum_{j=1}^N f_j(x_j) - \left\langle \lambda, \sum_{j=1}^N A_j x_j -b\right\rangle + \frac{\gamma}{2}\left\|\sum_{j=1}^N A_j x_j -b\right\|^2,\]
where $\lambda$ is the Lagrange multiplier and $\gamma > 0$ is a penalty parameter.
In a typical iteration of the ADMM for solving \eqref{prob:N}, the following updating procedure is implemented:
\be\label{admm-N}
\left\{\ba{lcl} x_1^{k+1} & := & \argmin_{x_1\in \XCal_1} \ \LCal_\gamma(x_1,x_2^k,\ldots,x_N^k;\lambda^k) \\
                x_2^{k+1} & := & \argmin_{x_2\in \XCal_2} \ \LCal_\gamma(x_1^{k+1},x_2,x_3^k,\ldots,x_N^k;\lambda^k) \\
                          & \vdots & \\
                x_N^{k+1} & := & \argmin_{x_N\in \XCal_N} \ \LCal_\gamma(x_1^{k+1},x_2^{k+1},\ldots,x_{N-1}^{k+1},x_N;\lambda^k) \\
                \lambda^{k+1} & := & \lambda^k - \gamma \left(\sum_{j=1}^N A_j x_j^{k+1} -b\right).           \ea\right. \ee

Note that the ADMM \eqref{admm-N} minimizes in each iteration the augmented Lagrangian function with respect to $x_1,\ldots,x_N$ alternatingly in a Gauss-Seidel manner.
The ADMM \eqref{admm-N} for solving two-block convex minimization problems (i.e., $N=2$) has been studied extensively in the literature.
The global convergence of ADMM \eqref{admm-N} when $N=2$ has been shown in \cite{Gabay-83,Eckstein-Bertsekas-1992}. There are also some recent works that study the convergence rate properties of ADMM when $N=2$ (see, e.g., \cite{He-Yuan-rate-ADM-2012,Monteiro-Svaiter-2010a,Deng-Yin-2012,Boley-2012,He-Yuan-nonergodic-2012}).

However, the convergence of multi-block ADMM \eqref{admm-N} (we call \eqref{admm-N} {\em multi-block ADMM} when $N\geq 3$) has remained unclear for a long time. Recently, Chen \etal\cite{Chen-admm-failure-2013} constructed a counterexample to show the failure of ADMM \eqref{admm-N} when $N\geq 3$.
Notwithstanding its theoretical convergence assurance, the multi-block ADMM \eqref{admm-N} has been applied very successfully to solve problems with $N\, (N\geq 3)$ block variables; for example, see \cite{Tao-Yuan-SPCP-2011,Wright-RASL-TPAMI}. It is thus of great interest to further study sufficient conditions that can guarantee the convergence of multi-block ADMM. Some recent works on studying the sufficient conditions guaranteeing the convergence of multi-block ADMM are described briefly as follows. Han and Yuan \cite{Han-Yuan-note-2012} showed that the multi-ADMM \eqref{admm-N} converges if all the functions $f_1,\ldots,f_N$ are strongly convex and $\gamma$ is restricted to certain region. This condition is relaxed in \cite{Chen-Shen-You-convergence-2013, Lin-Ma-Zhang-convergence-2014} to allow only $N-1$ functions to be strongly convex and $\gamma$ is restricted to certain region. Especially, Lin, Ma and Zhang \cite{Lin-Ma-Zhang-convergence-2014} proved the sublinear convergence rate under such conditions. Closely related to \cite{Chen-Shen-You-convergence-2013, Lin-Ma-Zhang-convergence-2014}, Cai, Han and Yuan \cite{Cai-Han-Yuan-direct-2014} and Li, Sun and Toh \cite{Li-Sun-Toh-convergent-2015} proved that for $N=3$, convergence of multi-block ADMM can be guaranteed under the assumption that only one function among $f_1$, $f_2$ and $f_3$ is required to be strongly convex, and $\gamma$ is restricted in certain region.
%In a very recent work \cite{Li-Sun-Toh-convergent-2015}, the authors study a variant of ADMM \eqref{admm-N} using a relation factor $\alpha\in(0,\frac{1+\sqrt{5}}{2})$ in the update of dual variable, and certain proximal terms $P_i\succ 0$ to regularize the $x_i$-subproblems. They show that this variant of ADMM \eqref{admm-N} converges with only one strongly convex objective function, and the penalty parameter $\gamma$ can be chosen {\color{blue} to take any value} in $(0,\infty)$. However, if the factor $\alpha$ is chosen to be close to $\frac{1+\sqrt{5}}{2}$, {\color{blue} then} $\gamma$ should be very close to $0$ and the smallest eigenvalue of $P_3$ must be sufficiently large. {\color{blue} This potentially slows down the convergence of the algorithm.} Moreover, \cite{Cai-Han-Yuan-direct-2014,Li-Sun-Toh-convergent-2015} focus on 3-block ADMM, and their extension to multi-block ADMM \eqref{admm-N} requires $N-2$ objective functions to be strongly convex.
In addition to strong convexity of $f_2,\ldots,f_N$, by assuming further conditions on the smoothness of the functions and some rank conditions on the matrices in the linear constraints, Lin, Ma and Zhang \cite{Lin-Ma-Zhang-2014-linear} proved the globally linear convergence of multi-block ADMM. Note that the above mentioned works all require that (parts of) the objective function is strongly convex.
Without assuming strong convexity, Hong and Luo \cite{Luo-ADMM-2012} studied a variant of ADMM \eqref{admm-N} with small stepsize in updating the Lagrangian multiplier. Specifically, \cite{Luo-ADMM-2012} proposes to replace the last equation in \eqref{admm-N} to
\[
\lambda^{k+1} := \lambda^k - \alpha\gamma \left(\sum_{j=1}^N A_j x_j^{k+1} -b\right),
\]
where $\alpha>0$ is a small step size. Linear convergence of this variant is proven under the assumption that the objective function satisfies certain error bound conditions. However, it is noted that the selection of $\alpha$ is in fact bounded by some parameters associated with the error bound conditions to guarantee the convergence. Therefore, it might be difficult to choose $\alpha$ in practice. There are also studies on the convergence and convergence rate of some other variants of ADMM \eqref{admm-N}, and we refer the interested readers to \cite{He-Tao-Yuan-2012,He-Tao-Yuan-MOR-2013,He-Hou-Yuan-Jacob-2013,Deng-admm-2014,Sun-Toh-Yang-4admm-2014,Hong-etal-2014-BSUMM,Wang-etal-2013-multi-2} for the details of these variants. However, it is observed by many researchers that modified versions of ADMM though with convergence guarantee, often perform slower than the multi-block ADMM with no convergent guarantee (see \cite{Sun-Toh-Yang-4admm-2014}). Therefore, in this paper, we focus on studying the sufficient conditions that guarantee the convergence of the direct extension of ADMM, i.e., the multi-block ADMM \eqref{admm-N} and studying its convergence rate.

{\bf Our contribution.} The main contribution in this paper lies in the following. First, we show that the ADMM \eqref{admm-N} when $N\geq 3$ returns an $\epsilon$-optimal solution within $O(1/\epsilon^2)$ iterations, with the condition that $\gamma$ depends on $\epsilon$. Here we do not assume strong convexity of any objective function $f_i$. It should be pointed out that our result does not contradict the counterexample proposed in \cite{Chen-admm-failure-2013} since we apply the ADMM \eqref{admm-N} to an associated perturbed problem of \eqref{prob:N} rather than \eqref{prob:N} itself. Secondly, we show that the ADMM \eqref{admm-N} when $N\geq 3$ returns an $\epsilon$-optimal solution within $O(1/\epsilon)$ iterations under the condition that the augmented Lagrangian $\LCal_\gamma$ is a Kurdyka-{\L}ojasiewicz (KL) function \cite{Bolte-characterizations-2010, Bolte-Sabach-Teboulle-2014}, $\nabla f_N$ is Lipschitz continuous, $A_N=I$, and $\gamma$ is sufficiently large. To the best of our knowledge, the convergence rate results given in this paper are the first sublinear convergence rate results for the unmodified multi-block ADMM without assuming any strong convexity of the objective function (note that although without assuming strong convexity, \cite{Luo-ADMM-2012} studies a variant of the multi-block ADMM). In this sense, the results presented in this paper complement with the existing results in the literature.

{\bf Organization.} The rest of this paper is organized as follows. In Section \ref{sec:pre} we provide some preliminaries for our convergence rate analysis. In Section \ref{sec:scenario-1}, we prove the $O(1/\epsilon^2)$ iteration complexity of ADMM \eqref{admm-N} by introducing an associated problem of \eqref{prob:N}. In Section \ref{sec:scenario-2}, we prove the $O(1/\epsilon)$ iteration complexity of ADMM \eqref{admm-N} with Kurdyka-{\L}ojasiewicz (KL) property.
%Section \ref{sec:conclusion} draws some conclusions and points out some future directions.
%{\color{red} (SZ: I commented out the conclusion section as it does not contain much information. But we can have it back if we have more things to say in the future.)}

\section{Preliminaries}\label{sec:pre}
We denote $\Omega = \XCal_1\times \ldots \times \XCal_N \times \RR^p$ and the optimal set of \eqref{prob:N} as $\Omega^*$, and the following assumption is made throughout this paper.

\begin{assumption}\label{assumption-1} The optimal set $\Omega^*$ for problem \eqref{prob:N} is non-empty.\end{assumption}
According to the first-order optimality conditions for \eqref{prob:N}, solving \eqref{prob:N} is equivalent to finding $$(x_1^*,x_2^*,\ldots,x_N^*,\lambda^*)\in\Omega^*$$
such that the following holds:
\begin{equation}\label{kkt}
%\begin{aligned}
\left\{
\begin{array}{l}
(x_i-x_i^*)^\top (g_i(x_i^*)-A_i^\top\lambda^*) \geq 0, \quad \forall x_i\in\XCal_i,\\
A_{1}x_{1}^{*} + \cdots +A_{N}x_{N}^{*}-b=0,
%\end{aligned}
\end{array}
\right.
\end{equation}
for $i=1,2,\ldots,N$.

In this paper, we analyze the iteration complexity of ADMM \eqref{admm-N} under two scenarios. The conditions of the two scenarios are listed in Table \ref{tab:N-scenarios}. The following assumption is only used in Scenario 2.
\begin{assumption}\label{assumption-2} We assume that $\XCal_N=\RR^{n_N}$. We also assume that $f_i$ has a finite lower bound, i.e., $\inf_{x_i\in\XCal_i} f_i(x_i)\geq f_i^*>-\infty$ for $i=1,2,\ldots,N$. Moreover, it is assumed that $f_i+\textbf{1}_{\XCal_i}$ is a coercive function for $i=1,2,\ldots,N-1$, where $\textbf{1}_{\XCal_i}$ denotes the indicator function of $\XCal_i$, i.e.,
\[\textbf{1}_{\XCal_i}(x_i) = \left\{\ba{ll} 0, & \mbox{ if } x_i\in \XCal_i \\ +\infty, & \mbox{ otherwise. }\ea\right.\]
Furthermore, we assume that $\LCal_\gamma$ is a KL function (will be defined later). \end{assumption}

\begin{table}[htdp]
\begin{tabular}{c|c|c|c|c}\hline
Scenario & Lipschitz Continuous & Matrices & Additional Assumption & Iteration Complexity \\\hline\hline
1 & --- & --- & $\frac{\epsilon}{2}\leq\gamma\leq\epsilon$ & $O(1/\epsilon^2)$ \\ \hline
2 & $\nabla f_N$ & $A_N=I$ & $\gamma>\sqrt{2}L$ and Assumption \ref{assumption-2} &  $O(1/\epsilon)$ \\ \hline
\end{tabular}
\caption{Two Scenarios Leading to Sublinear Convergence}\label{tab:N-scenarios}
\end{table}

\begin{remark}
Some remarks are in order here regarding the conditions in Scenario 2. Note that it is not very restrictive to require $f_i+\textbf{1}_{\XCal_i}$ to be a coercive function. In fact, many functions used as regularization terms including $\ell_1$-norm, $\ell_2$-norm, $\ell_\infty$-norm for vectors and nuclear norm for matrices are all coercive functions; assuming the compactness of $\XCal_i$ also leads to the coerciveness of $f_i+\textbf{1}_{\XCal_i}$. Moreover, the assumptions $A_N=I$ and $\nabla f_N$ is Lipschitz continuous actually cover many interesting applications in practice. For example, many problems arising from machine learning, statistics, image processing and so on always have the following structure:
\be\label{prob:fN-special}
\min \ f_1(x_1) + \cdots + f_{N-1}(x_{N-1}) + f_N(b-A_1x_1- \cdots - A_{N-1}x_{N-1}),
\ee
where $f_N$ denotes a loss function on data fitting, which is usually a smooth function, and $f_1, \ldots, f_{N-1}$ are regularization terms to promote certain structures of the solution. This problem is usually referred as {\it sharing problem} (see, e.g., \cite{Boyd-etal-ADM-survey-2011,Hong-nonconvex-admm-2014}). \eqref{prob:fN-special} can be reformulated as
\be\label{prob:fN-special-admm}
\ba{ll}
\min & f_1(x_1) + \cdots + f_{N-1}(x_{N-1}) + f_N(x_N) \\
\st  & A_1x_1+ \cdots + A_{N-1}x_{N-1} + x_N = b,
\ea
\ee
which is in the form of \eqref{prob:N} and can be solved by ADMM (see \cite{Boyd-etal-ADM-survey-2011,Hong-nonconvex-admm-2014}). Note that $A_N=I$ in \eqref{prob:fN-special-admm} and it is very natural to assume that $\nabla f_N$ is Lipschitz continuous. Thus the conditions in Scenario 2 are satisfied.
\end{remark}

{\bf Notations. }
For simplicity, we use the following notation to denote the stacked vectors or tuples:
\[ u = \left(\begin{array}{c} x_{1} \\ \vdots \\ x_{N} \end{array} \right),
u^k = \left(\begin{array}{c} x_{1}^k \\ \vdots \\ x_{N}^k \end{array} \right),
u^* = \left(\begin{array}{c} x_{1}^* \\ \vdots \\ x_{N}^* \end{array} \right),\\
w = \left(\begin{array}{c} u \\ \lambda \end{array} \right),
w^k = \left(\begin{array}{c} u^k \\ \lambda^k \end{array} \right),
w^* = \left(\begin{array}{c} u^* \\ \lambda^* \end{array} \right).\]
We denote by $f(u)\equiv f_1(x_{1})+\cdots+f_{N}(x_{N})$ the objective function of problem \eqref{prob:N}; $\mathbf{1}_{\XCal}$ is the indicator function of $\XCal$; $\nabla f$ is the gradient of $f$; $\|x\|$ denotes the Euclidean norm of $x$.

In our analysis, the following two well-known identities are used frequently,
\begin{eqnarray}
 (w_1-w_2)^{\top}(w_3-w_4) &=& \frac{1}{2}\left(\|w_1-w_4\|^{2}-\|w_1-w_3\|^{2}\right)+\frac{1}{2}\left(\|w_3-w_2\|^{2}-\|w_4-w_2\|^{2}\right), \label{identity-4} \\
 (w_{1}-w_{2})^\top(w_{3}-w_{1}) &=& \frac{1}{2}\left(\|w_{2}-w_{3}\|^{2}-\|w_{1}-w_{2}\|^{2}-\|w_{1}-w_{3}\|^{2}\right). \label{identity-3}
\end{eqnarray}

\section{Iteration Complexity of ADMM: Associated Perturbation}\label{sec:scenario-1}
In this section, we prove the $O(1/\epsilon^2)$ iteration complexity of ADMM  \eqref{admm-N} under the conditions in Scenario 1 of Table \ref{tab:N-scenarios}. Indeed, given $\epsilon>0$ sufficiently small and initial point $u^0$, we introduce an associated perturbed problem of \eqref{prob:N}, i.e.,
\be\label{prob:N-Perturb}\ba{ll} \min & f_1(x_1) + \tilde{f}_2(x_2) + \cdots + \tilde{f}_N(x_N) \\
                         \st  & A_1 x_1 + A_2 x_2 + \cdots + A_N x_N = b \\
                              & x_i\in\XCal_i, i=1,\ldots,N, \ea \ee
where $\tilde{f}_i(x_i) = f_i(x_i) + \frac{\mu}{2}\left\|A_i x_i - A_i x_i^0\right\|^2$ for $i=2,\ldots,N$, and
$\mu = \epsilon(N-2)(N+1)$.
Note $\tilde{f}_i$ are not necessarily strongly convex. We prove that the ADMM \eqref{admm-N} for associated perturbed problem \eqref{prob:N-Perturb} returns an $\epsilon$-optimal solution of the original problem \eqref{prob:N}, in terms of both objective value and constraint violation, within $O(1/\epsilon^2)$ iterations.

The ADMM for solving \eqref{prob:N-Perturb} can be summarized as (note that some constant terms in the subproblems are discarded):
\begin{eqnarray}
x_{1}^{k+1} &:=& \argmin_{x_{1} \in \XCal_{1}}f_{1}(x_{1})+\frac{\gamma}{2}\left\|A_{1}x_{1} + \sum\limits_{j=2}^N A_{j}x_{j}^{k}-b-\frac{1}{\gamma}\lambda^{k}\right\|^{2}, \label{scenario-1-update-x-1}\\
x_{i}^{k+1} &:=& \argmin_{x_{i} \in \XCal_{i}}\tilde{f}_{i}(x_{i}) +\frac{\gamma}{2}\left\| \sum\limits_{j=1}^{i-1} A_{j}x_{j}^{k+1} + A_{i}x_{i} + \sum\limits_{j=i+1}^{N} A_{j}x_{j}^{k}-b-\frac{1}{\gamma}\lambda^{k}\right\|^{2}, \ i = 2,\ldots,N,  \label{scenario-1-update-x-i}\\
\lambda^{k+1} &:=& \lambda^{k} - \gamma \left(A_{1}x_{1}^{k+1}+A_{2}x_{2}^{k+1}+ \cdots + A_{N}x_{N}^{k+1}-b\right). \label{scenario-1-update-lambda}
\end{eqnarray}
The first-order optimality conditions for \eqref{scenario-1-update-x-1}-\eqref{scenario-1-update-x-i} are given respectively by $x_{i}^{k+1}\in\XCal_{i}$ and
\begin{align}
& (x_{1}-x_{1}^{k+1})^{\top}\left[
g_1(x_{1}^{k+1})-A_{1}^{\top}\lambda^{k}+\gamma A_{1}^{\top}\left(A_{1}x_{1}^{k+1}+\sum\limits_{j=2}^N A_{j}x_{j}^{k}-b\right)\right]\geq 0, \label{scenario-1-opt-x-1} \\
& (x_{i}-x_{i}^{k+1})^{\top}\left[ g_i(x_{i}^{k+1})+\mu A_i^\top A_i\left(x_i^{k+1}-x_i^0\right)-A_{i}^{\top}\lambda^{k}+\gamma A_{i}^{\top}\left(\sum\limits_{j=1}^{i} A_{j}x_{j}^{k+1} + \sum\limits_{j=i+1}^{N} A_{j}x_{j}^{k}-b\right)\right] \geq 0,\label{scenario-1-opt-x-i}
\end{align}
hold for any $x_i\in\XCal_i$ and $g_i \in \partial f_i$, a subgradient of $f_i$, for $i=1,2,\ldots,N$. Moreover, by combining with \eqref{scenario-1-update-lambda}, \eqref{scenario-1-opt-x-1}-\eqref{scenario-1-opt-x-i} can be rewritten as
\begin{align}
& (x_{1}-x_{1}^{k+1})^{\top}\left[ g_{1}(x_{1}^{k+1})-A_{1}^{\top}\lambda^{k+1} +\gamma A_{1}^{\top}\left(\sum\limits_{j=2}^N A_{j}(x_{j}^{k}-x_{j}^{k+1})\right)\right] \geq 0, \label{scenario-1-opt-x-1-lambda} \\
& (x_{i}-x_{i}^{k+1})^{\top}\left[ g_i(x_{i}^{k+1})+\mu A_i^\top A_i\left( x_i^{k+1}- x_i^0\right)-A_{i}^{\top}\lambda^{k+1}+\gamma A_{i}^{\top}\left(\sum\limits_{j=i+1}^N A_{j}(x_{j}^{k}-x_{j}^{k+1})\right)\right] \geq 0. \label{scenario-1-opt-x-i-lambda}
\end{align}

\begin{lemma}\label{lemma-associated-perturbation}
Let $(x_1^{k+1},x_{2}^{k+1},\ldots,x_N^{k+1},\lambda^{k+1})\in\Omega$ be generated by the ADMM \eqref{admm-N} from given $(x_{2}^{k},\ldots,x_{N}^{k},\lambda^{k})$. For any $u^*=(x_1^*,x_2^*,\ldots,x_N^*)\in\Omega^*$ and $\lambda\in\RR^p$, it holds true under conditions in Scenario 1 that
\begin{eqnarray}
& & f(u^*)-f(u^{k+1})+\left(\begin{array}{c} x_1^*-x_1^{k+1} \\ x_2^*-x_2^{k+1} \\ \vdots \\ x_N^*-x_{N}^{k+1} \\ \lambda-\lambda^{k+1}\end{array} \right)^{\top}
\left(\begin{array}{c} -A_1^\top\lambda^{k+1} \\ -A_2^\top\lambda^{k+1} \\ \vdots \\ -A_N^\top\lambda^{k+1} \\ \sum_{i=1}^N A_i x_i^{k+1}-b \end{array} \right) \nonumber \\
& & + \frac{1}{2\gamma}\left(\left\|\lambda-\lambda^{k}\right\|^{2}-\left\|\lambda-\lambda^{k+1}\right\|^{2}\right) + \frac{\epsilon(N-2)(N+1)}{2}\sum\limits_{i=2}^N \left\| A_i x_i^* - A_i x_i^{0}\right\|^{2} \nonumber  \\
& & + \frac{\gamma}{2}\sum\limits_{i=1}^{N-1} \left(\left\| \sum\limits_{j=1}^i A_{j}x_{j}^* + \sum\limits_{j=i+1}^N A_j x_j^{k} - b\right\|^{2}-\left\| \sum\limits_{j=1}^i A_{j}x_{j}^* + \sum\limits_{j=i+1}^N A_j x_j^{k+1} - b\right\|^{2}\right) \nonumber \\
& \geq & 0.\label{inequality-associated-perturbation}
\end{eqnarray}
\end{lemma}

\begin{proof}
Note that combining \eqref{scenario-1-opt-x-1-lambda}-\eqref{scenario-1-opt-x-i-lambda} yields
\begin{eqnarray}\label{inequality-associated-perturbation-1}
& & \left(\begin{array}{c} x_{1}-x_{1}^{k+1} \\ x_{2}-x_{2}^{k+1} \\ \vdots \\ x_{N}-x_{N}^{k+1} \end{array} \right)^{\top}\left[
\left(\begin{array}{c} g_1(x_{1}^{k+1})-A_{1}^{\top}\lambda^{k+1} \\ g_{2}(x_{2}^{k+1})-A_{2}^{\top}\lambda^{k+1} \\ \vdots \\ g_{N}(x_{N}^{k+1})-A_{N}^{\top}\lambda^{k+1} \end{array} \right) + \left(\begin{array}{c} 0 \\ \mu A_2^\top(A_2 x_2^{k+1}-A_2 x_2^0) \\ \vdots \\ \mu A_N^\top(A_N x_N^{k+1}-A_N x_N^0) \end{array} \right) + H\left(\begin{array}{c} x_{2}^{k}-x_{2}^{k+1} \\ \vdots \\ x_{N}^{k}-x_{N}^{k+1} \end{array} \right)\right] \nonumber \\
& \geq & 0,
\end{eqnarray}
where $H\in \RR^{\left(\sum_{i=1}^N n_i\right) \times \left(\sum_{i=2}^N n_i\right)}$ is defined as follow:
\begin{displaymath}
H := \left(\begin{array}{cccc} \gamma A_{1}^{\top}A_{2} & \gamma A_{1}^{\top}A_{3} & \cdots & \gamma A_{1}^{\top}A_{N} \\ 0 & \gamma A_{2}^{\top}A_{3} & \cdots & \gamma A_{2}^{\top}A_{N} \\ \vdots & \ddots & \ddots & \vdots \\ 0 & 0 & \cdots & \gamma A_{N-1}^{\top}A_{N} \\ 0 & 0 & \cdots & 0 \end{array} \right) .
\end{displaymath}
The key step in our proof is to bound the following terms
\begin{displaymath}
(x_{i}-x_{i}^{k+1})^{\top}A_{i}^{\top}\left( \sum\limits_{j=i+1}^N A_{j}(x_{j}^{k}-x_{j}^{k+1})\right), \ i=1,2,\ldots, N-1.
\end{displaymath}

%For the first term, we have
%\begin{eqnarray*}
%& & (x_1-x_1^{k+1})^\top A_1^\top \left[ \sum\limits_{j=2}^N A_j (x_j^k - x_j^{k+1})\right] \\
%& = & \left[(A_1 x_1 - b)-(A_1 x_1^{k+1} - b)\right]^\top \left[\left(-\sum\limits_{j=2}^N A_j x_j^{k+1}\right)-\left(-\sum\limits_{j=2}^N A_j x_j^k\right)\right] \\
%& = & \frac{1}{2}\left(\left\| A_1 x_1 + \sum\limits_{j=2}^N A_j x_j^k - b\right\|^2-\left\| A_1 x_1 + \sum\limits_{j=2}^N A_j x_j^{k+1} - b \right\|^2 \right)\\
%&  & + \frac{1}{2}\left(\left\| \sum\limits_{j=1}^N A_j x_j^{k+1} - b\right\|^2 - \left\| A_1 x_1^{k+1} + \sum\limits_{j=2}^N A_j x_j^k - b\right\|^2\right)\\
%& \leq & \frac{1}{2}\left(\left\| A_1 x_1 +\sum\limits_{j=2}^N A_j x_j^k - b\right\|^2 - \left\| A_1 x_1 + \sum\limits_{j=2}^N A_j x_j^{k+1} - b\right\|^2 \right) + \frac{1}{2\gamma^2}\|\lambda^k -\lambda^{k+1}\|^2,
%\end{eqnarray*}
%where in the second equality we used the identity \eqref{identity-4}, and the last equality follows from the updating formula for $\lambda^{k+1}$ in \eqref{scenario-1-update-lambda}.

For $i=1,2,\ldots,N-1$, we have,
\begin{eqnarray*}
& & (x_i -x_i^{k+1})^\top A_i^\top \left( \sum\limits_{j=i+1}^N A_j (x_j^k - x_j^{k+1})\right) \\
& = & \left[\left(\sum\limits_{j=1}^i A_j x_j - b\right) - \left(\sum\limits_{j=1}^{i-1} A_j x_j + A_i x_i^{k+1}-b\right)\right]^\top \left[\left(-\sum\limits_{j=i+1}^N A_j x_j^{k+1}\right)-\left(-\sum\limits_{j=i+1}^N A_j x_j^k\right)\right]\\
& = & \frac{1}{2}\left(\left\| \sum\limits_{j=1}^i A_j x_j + \sum\limits_{j=i+1}^N A_j x_j^k - b\right\|^{2}-\left\| \sum\limits_{j=1}^i A_j x_j + \sum\limits_{j=i+1}^N A_j x_j^{k+1} - b\right\|^{2}\right)\\
&   & +\frac{1}{2}\left(\left\| \sum\limits_{j=1}^{i-1} A_j x_j + \sum\limits_{j=i}^N A_j x_j^{k+1} - b\right\|^{2} - \left\| \sum\limits_{j=1}^{i-1} A_j x_j + A_i x_i^{k+1}+\sum\limits_{j=i+1}^N A_j x_j^k - b\right\|^{2}\right)\\
&\leq & \frac{1}{2}\left(\left\| \sum\limits_{j=1}^i A_j x_j + \sum\limits_{j=i+1}^N A_j x_j^k - b\right\|^{2}-\left\| \sum\limits_{j=1}^i A_j x_j + \sum\limits_{j=i+1}^N A_j x_j^{k+1} - b\right\|^{2}\right)\\
& & +\frac{1}{2}\left\| \sum\limits_{j=1}^{i-1} A_j x_j + \sum\limits_{j=i}^N A_j x_j^{k+1} - b \right\|^{2},
\end{eqnarray*}
where in the second equality we applied the identity \eqref{identity-4}.

Therefore, we have
\begin{eqnarray}
& &  \left(\begin{array}{c} x_1 - x_1^{k+1} \\ x_2 - x_2^{k+1} \\ \vdots \\ x_N - x_N^{k+1} \end{array} \right)^{\top} \left(\begin{array}{cccc} \gamma A_1^\top A_2 & \gamma A_1^\top A_3 & \cdots & \gamma A_1^\top A_N \\ 0 & \gamma A_2^\top A_3 & \cdots & \gamma A_2^\top A_N \\ \vdots & \ddots & \ddots & \vdots \\ 0 & 0 & \cdots & \gamma A_{N-1}^\top A_N \\ 0 & 0 & \cdots & 0 \end{array} \right)\left(\begin{array}{c} x_2^k - x_2^{k+1} \\ \vdots \\ x_N^k - x_N^{k+1} \end{array} \right) \nonumber \\
& \leq & \frac{\gamma}{2}\sum\limits_{i=1}^{N-1} \left(\left\| \sum\limits_{j=1}^i A_j x_j + \sum\limits_{j=i+1}^N A_j x_j^k - b\right\|^2 - \left\| \sum\limits_{j=1}^i A_j x_j + \sum\limits_{j=i+1}^N A_j x_j^{k+1} - b\right\|^2\right) \nonumber \\
 & & + \frac{1}{2\gamma}\left\|\lambda^{k+1}-\lambda^k\right\|^{2}+ \frac{\gamma}{2}\sum\limits_{i=2}^{N-1}\left\| \sum\limits_{j=1}^{i-1} A_j x_j +\sum\limits_{j=i}^N A_j x_j^{k+1}-b\right\|^2 . \label{inequality-associated-perturbation-2}
\end{eqnarray}
Combining \eqref{scenario-1-update-lambda}, \eqref{inequality-associated-perturbation-1} and \eqref{inequality-associated-perturbation-2}, it holds for any $\lambda\in\RR^p$ that
\begin{eqnarray}
& & \left(\begin{array}{c} x_1 - x_1^{k+1} \\ x_2 - x_2^{k+1} \\ \vdots \\ x_N - x_N^{k+1} \\ \lambda-\lambda^{k+1}\end{array} \right)^{\top}
\left(\begin{array}{c} g_1(x_1^{k+1}) - A_1^\top\lambda^{k+1} \\ g_2(x_2^{k+1}) - A_2^\top\lambda^{k+1} \\ \vdots \\ g_N(x_N^{k+1})-A_N^\top\lambda^{k+1} \\ \sum_{i=1}^N A_i x_i^{k+1}-b \end{array} \right) + \frac{1}{\gamma}\left( \lambda - \lambda^{k+1}\right)^\top\left( \lambda^{k+1} - \lambda^k\right) \nonumber \\
& & + \mu\sum\limits_{i=2}^N \left(x_i - x_i^{k+1}\right)^\top A_i^\top A_i\left( x_i^{k+1} - x_i^0\right) + \frac{1}{2\gamma}\left\|\lambda^{k+1} - \lambda^k \right\|^2 + \frac{\gamma}{2}\sum\limits_{i=2}^{N-1}\left\| \sum\limits_{j=1}^{i-1} A_j x_j + \sum\limits_{j=i}^N A_j x_j^{k+1} - b\right\|^2 \nonumber \\
& & + \frac{\gamma}{2}\sum\limits_{i=1}^{N-1} \left(\left\| \sum\limits_{j=1}^i A_j x_j + \sum\limits_{j=i+1}^N A_j x_j^k - b\right\|^2 -\left\| \sum\limits_{j=1}^i A_j x_j + \sum\limits_{j=i+1}^N A_j x_j^{k+1} - b\right\|^2\right) \nonumber \\
& \geq & 0. \label{inequality-associated-perturbation-3}
\end{eqnarray}
Using \eqref{identity-3}, we have
\begin{eqnarray*}
\frac{1}{\gamma}\left(\lambda-\lambda^{k+1}\right)^{\top}\left( \lambda^{k+1} - \lambda^k \right)+\frac{1}{2\gamma}\left\|\lambda^{k+1} - \lambda^k \right\|^2 = \frac{1}{2\gamma}\left(\left\|\lambda-\lambda^k \right\|^2 - \left\|\lambda-\lambda^{k+1}\right\|^2\right),
\end{eqnarray*}
and
\begin{eqnarray*}
& & \mu\left(x_i - x_i^{k+1}\right)^\top A_j^\top A_j\left( x_i^{k+1} - x_i^0 \right)  \\
& = & \frac{\mu}{2}\left( \left\| A_i x_i - A_i x_i^0 \right\|^2 - \left\| A_i x_i^{k+1} - A_i x_i^0 \right\|^{2} -  \left\| A_i x_i - A_i x_i^{k+1}\right\|^2\right) \\
& \leq & \frac{\mu}{2}\left\| A_i x_i - A_i x_i^0 \right\|^2 -  \frac{\mu}{2}\left\| A_i x_i - A_i x_i^{k+1} \right\|^2.
\end{eqnarray*}
Letting $u=u^{*}$ in \eqref{inequality-associated-perturbation-3}, and invoking the convexity of $f_i$ that
\begin{eqnarray*}
f_{i}(x_{i}^{*})-f_{i}(x_{i}^{k+1})\geq (x_{i}^{*}-x_{i}^{k+1})^{\top}
g_{i}(x_{i}^{k+1}),\quad i=1,2,\ldots,N
\end{eqnarray*}
and
\begin{eqnarray*}
\frac{\gamma}{2}\sum\limits_{i=2}^{N-1}\left\| \sum\limits_{j=1}^{i-1} A_j x_j^* + \sum\limits_{j=i}^N A_j x_j^{k+1} - b\right\|^2 & = & \frac{\gamma}{2}\sum\limits_{i=2}^{N-1}\left\| \sum\limits_{j=i}^N A_j (x_j^{k+1} - x_j^*)\right\|^2 \\
& \leq & \frac{\gamma(N+1)(N-2)}{2}\sum\limits_{i=2}^{N} \left\| A_i x_i^{k+1} - A_i x_i^*\right\|^2,
\end{eqnarray*}
we obtain,
\begin{eqnarray*}
& & f(u^*)-f(u^{k+1})+\left(\begin{array}{c} x_1^* - x_1^{k+1} \\ x_2^* - x_2^{k+1} \\ \vdots \\ x_N^* - x_N^{k+1} \\ \lambda-\lambda^{k+1}\end{array} \right)^\top
\left(\begin{array}{c} -A_1^\top\lambda^{k+1} \\ -A_2^\top\lambda^{k+1} \\ \vdots \\ -A_N^\top\lambda^{k+1} \\  \sum_{i=1}^N A_i x_i^{k+1} - b \end{array} \right)\\
& & + \frac{1}{2\gamma}\left( \left\|\lambda-\lambda^k \right\|^2 - \left\|\lambda-\lambda^{k+1}\right\|^2 \right) + \frac{\mu}{2}\sum\limits_{i=2}^N \left( \left\| A_i x_i^* - A_i x_i^0\right\|^2 - \left\| A_i x_i^* - A_i x_i^{k+1} \right\|^2 \right) \\
& & + \frac{\gamma}{2}\sum\limits_{i=1}^{N-1} \left(\left\| \sum\limits_{j=1}^i A_j x_j^* + \sum\limits_{j=i+1}^N A_j x_j^k - b\right\|^2 - \left\| \sum\limits_{j=1}^i A_j x_j^* + \sum\limits_{j=i+1}^N A_j x_j^{k+1} - b \right\|^{2}\right) \\
& & +\frac{\gamma(N+1)(N-2)}{2}\sum\limits_{i=2}^{N} \left\| A_i x_i^* - A_i x_i^{k+1}\right\|^{2} \\
& \geq & 0.
\end{eqnarray*}
This together with the facts that $\mu = \epsilon(N-2)(N+1)$ and $\gamma\leq\epsilon$ implies that
\begin{displaymath}
\frac{\gamma(N+1)(N-2)}{2}\sum\limits_{j=2}^{N} \left\| A_j x_j^* - A_j x_j^{k+1}\right\|^{2} - \frac{\mu}{2}\sum\limits_{j=2}^{N} \left\| A_j x_j^* - A_j x_j^{k+1}\right\|^{2}\leq 0,
\end{displaymath}
which further implies the desired inequality \eqref{inequality-associated-perturbation}.
\end{proof}

Now we are ready to prove the $O(1/\epsilon^2)$ iteration complexity of the ADMM for \eqref{prob:N} in an ergodic case.

\begin{theorem}\label{thm-scenario-1-ergodic-N}
Let $(x_1^{k+1},x_2^{k+1},\ldots,x_N^{k+1},\lambda^{k+1})\in\Omega$ be generated by ADMM \eqref{scenario-1-update-x-1}-\eqref{scenario-1-update-lambda} from given $(x_2^k,\ldots,x_N^k,\lambda^k)$. For any integer $t>0$, let $\bar{u}^t = (\bar{x}_1^t, \bar{x}_2^t, \ldots, \bar{x}_N^t)$ and $\bar{\lambda}^t$ be defined as
\begin{eqnarray*}
\bar{x}_{i}^{t}=\frac{1}{t+1}\sum\limits_{k=0}^{t}x_{i}^{k+1}, \ i=1,2,\ldots,N, \quad
\bar{\lambda}^{t}=\frac{1}{t+1}\sum\limits_{k=0}^{t}\lambda^{k+1}.
\end{eqnarray*}
For any $(u^*,\lambda^*) \in\Omega^{*}$, by defining $\rho:=\|\lambda^*\|+1$, it holds in Scenario 1 that,
\begin{eqnarray*}
 0 & \leq & f(\bar{u}^t)-f(u^*)+\rho\left\| \sum\limits_{i=1}^N A_i\bar{x}_i^t-b \right\| \\
\nonumber & \leq & \frac{\rho^2 + \|\lambda^0\|^2}{\gamma (t+1)} + \frac{\gamma}{2(t+1)}\sum\limits_{i=1}^{N-1}\left\| \sum\limits_{j=i+1}^N A_j (x_j^0 - x_j^*)\right\|^{2} + \frac{\epsilon(N-2)(N+1)}{2}\sum\limits_{i=2}^N \left\| A_i x_i^* - A_i x_i^0 \right\|^{2}.
\end{eqnarray*}
This also implies that when $t=O(1/\epsilon^2)$, $\bar{u}^{t}=(\bar{x}_{1}^{t}, \bar{x}_{2}^{t}, \ldots, \bar{x}_{N}^{t})$ is an $\epsilon$-optimal solution to the original problem \eqref{prob:N}, i.e., both the error of the objective function value and the residual of the equality constraint satisfy that
\be\label{scenario-1-ergodic-N}|f(\bar{u}^t)-f(u^*)| = O(\epsilon), \quad \mbox{ and } \quad \left\| \sum\limits_{i=1}^N A_i\bar{x}_i^t-b \right\|=O(\epsilon).\ee
\end{theorem}

\begin{proof}
Because $(u^{k},\lambda^k)\in\Omega$, it holds that $(\bar{u}^{t},\bar{\lambda}^t)\in\Omega$ for all $t\geq 0$.
By Lemma \ref{lemma-associated-perturbation} and invoking the convexity of function $f(\cdot)$, we have
\begin{eqnarray}\label{long-eq-page-10}
&   & f(u^*)-f(\bar{u}^t) + \lambda^\top \left( \sum\limits_{i=1}^N A_i\bar{x}_i^t-b \right) \nonumber \\
& = & f(u^*)-f(\bar{u}^t)+\left(\begin{array}{c} x_1^* - \bar{x}_1^t \\ x_2^* - \bar{x}_2^t \\ \vdots \\ x_N^* -\bar{x}_N^t \\ \lambda-\bar{\lambda}^t\end{array} \right)^{\top}
\left(\begin{array}{c} -A_1^\top\bar{\lambda}^t \\ -A_2^\top\bar{\lambda}^t \\ \vdots \\ -A_N^\top\bar{\lambda}^t \\ \sum_{i=1}^N A_i \bar{x}_i^t - b \end{array} \right) \nonumber \\
& \geq & \frac{1}{t+1}\sum\limits_{k=0}^{t}\left[f(u^*) - f(u^{k+1}) +\left(\begin{array}{c} x_1^* - x_1^{k+1} \\ x_2^* - x_2^{k+1} \\ \vdots \\ x_N^* - x_N^{k+1} \\ \lambda-\lambda^{k+1}\end{array} \right)^\top
\left(\begin{array}{c} -A_1^\top\lambda^{k+1} \\ -A_2^\top\lambda^{k+1} \\ \vdots \\ -A_N^\top\lambda^{k+1} \\ \sum_{i=1}^N A_i x_i^{k+1}-b \end{array} \right)\right] \nonumber\\
& \geq &\frac{1}{t+1}\sum\limits_{k=0}^t\left[ \frac{1}{2\gamma}\left(\left\|\lambda - \lambda^{k+1}\right\|^2 - \left\| \lambda-\lambda^k \right\|^2 \right) - \frac{\epsilon(N-2)(N+1)}{2}\sum\limits_{i=2}^N \left\| A_i x_i^* - A_i x_i^0 \right\|^2 \right. \nonumber\\
& & \left. +\frac{\gamma}{2}\sum\limits_{i=1}^{N-1} \left(\left\| \sum\limits_{j=1}^i A_j x_j^* + \sum\limits_{j=i+1}^N A_j x_j^{k+1} - b\right\|^2 - \left\| \sum\limits_{j=1}^i A_j x_j^* + \sum\limits_{j=i+1}^N A_j x_j^k - b\right\|^2\right)\right]\nonumber\\
& \geq & -\frac{1}{2\gamma(t+1)}\left\|\lambda - \lambda^0 \right\|^2 -\frac{\gamma}{2(t+1)}\sum\limits_{i=1}^{N-1}\left\| \sum\limits_{j=1}^i A_j x_j^* + \sum\limits_{j=i+1}^N A_j x_j^0 - b\right\|^2 \nonumber\\
& & - \frac{\epsilon(N-2)(N+1)}{2}\sum\limits_{i=2}^N \left\| A_i x_i^* - A_i x_i^0 \right\|^2.
\end{eqnarray}
Note that this inequality holds for all $\lambda\in\RR^{p}$. From the optimality condition \eqref{kkt} we obtain
\[0 \geq f(u^*)-f(\bar{u}^t) + (\lambda^*)^\top \left( \sum\limits_{i=1}^N A_i\bar{x}_i^t-b \right).\]
Moreover, since $\rho:=\|\lambda^*\|+1$, by applying Cauchy-Schwarz inequality, we obtain
\begin{eqnarray}\label{scenario-1-ergodic-N-inequality-1}
0 \leq f(\bar{u}^t)-f(u^*)+\rho\left\| \sum\limits_{i=1}^N A_i\bar{x}_i^t-b \right\|. 
\end{eqnarray}
By setting $\lambda = -\rho \left(\sum_{i=1}^N A_i\bar{x}_i^t-b\right) /\left\|\sum_{i=1}^N A_i\bar{x}_i^t-b\right\|$ in \eqref{long-eq-page-10}, and noting that $\|\lambda\| = \rho$, we obtain
%since $\rho:=\|\lambda^*\|+1$,
%$\| \lambda -\lambda_0\|^2 \le 2(\rho^{2}+\|\lambda^0\|^2)$ for all $\|\lambda\|\leq\rho$, and $\sum_{i=1}^N A_i x^*_i=b$, we obtain
\begin{eqnarray}\label{scenario-1-ergodic-N-inequality-2}
& & f(\bar{u}^t)-f(u^*)+\rho\left\| \sum\limits_{i=1}^N A_i\bar{x}_i^t-b \right\| \\
\nonumber & \leq & \frac{\rho^2 + \|\lambda^0\|^2}{\gamma (t+1)} + \frac{\gamma}{2(t+1)}\sum\limits_{i=1}^{N-1}\left\| \sum\limits_{j=i+1}^N A_j (x_j^0 - x_j^*)\right\|^{2} + \frac{\epsilon(N-2)(N+1)}{2}\sum\limits_{i=2}^N \left\| A_i x_i^* - A_i x_i^{0} \right\|^{2}.
\end{eqnarray}
When $t=O(1/\epsilon^2)$, and together with the condition that $\frac{\epsilon}{2}\leq\gamma\leq\epsilon$, we have
\begin{eqnarray}\label{scenario-1-error-order}
\frac{\rho^2 + \|\lambda^0\|^2}{\gamma (t+1)} + \frac{\gamma}{2(t+1)}\sum\limits_{i=1}^{N-1}\left\| \sum\limits_{j=i+1}^N A_j (x_j^0 - x_j^*)\right\|^{2} + \frac{\epsilon(N-2)(N+1)}{2}\sum\limits_{i=2}^N \left\| A_i x_i^* - A_i x_i^0 \right\|^{2} = O(\epsilon).
\end{eqnarray}
%where we use Assumption \ref{assumption-2} that
%\begin{displaymath}
%\sum\limits_{i=2}^N \left\| A_i x_i^* - A_i x_i^0 \right\|^{2}\leq D^2 \sum\limits_{i=2}^N \left\| A_i \right\|^{2}
%\end{displaymath}
%is a constant which does not depend on $\epsilon$.

We now define the function
\[v(\xi) = \min \{ f(u) | \sum_{i=1}^N A_i x_i - b = \xi, x_i\in \XCal_i, i=1,2,\ldots,N \}.\]
It is easy to verify that $v$ is convex, $v(0)=f(u^*)$, and $\lambda^* \in \partial v(0)$.
Therefore, from the convexity of $v$, it holds that
\be\label{scenario-1-v-convex} v(\xi) \ge v(0) + \langle \lambda^*, \xi \rangle \ge f(u^*) - \| \lambda^*\| \|\xi\|.\ee
Let $\bar{\xi} = \sum\limits_{i=1}^N A_i \bar{x}_i^t - b$, we have $f(\bar{u}^t) \ge v(\bar{\xi})$.
Therefore, combining \eqref{scenario-1-ergodic-N-inequality-1}, \eqref{scenario-1-error-order} and \eqref{scenario-1-v-convex}, we get
\begin{eqnarray*}
& & - \| \lambda^*\| \| \bar{\xi} \| \leq f(\bar{u}^t) - f(u^*) \\
& \leq & \frac{\rho^2 + \|\lambda^0\|^2}{\gamma (t+1)} + \frac{\gamma}{2(t+1)}\sum\limits_{i=1}^{N-1}\left\| \sum\limits_{j=i+1}^N A_j (x_j^0 - x_j^*)\right\|^{2} + \frac{\epsilon(N-2)(N+1)}{2}\sum\limits_{i=2}^N \left\| A_i x_i^* - A_i x_i^0 \right\|^{2} - \rho \| \bar{\xi}\| \\
& \leq & C\epsilon - \rho \| \bar{\xi}\|,
\end{eqnarray*}
which, by using $\rho=\|\lambda^*\|+1$, yields,
\begin{eqnarray}
\label{scenario-1-infeasibility}
\| \sum\limits_{i=1}^N A_i \bar{x}_i^t  - b\| = \|\bar{\xi}\| \leq C\epsilon.
\end{eqnarray}
Moreover, by combining \eqref{scenario-1-ergodic-N-inequality-1} and \eqref{scenario-1-infeasibility}, one obtains that
\begin{eqnarray}\label{scenario-1-obj-diff}
-\rho C\epsilon \leq -\rho\|\bar{\xi}\| \leq f(\bar{u}^t) - f(u^*) \leq (1-\rho)C\epsilon .
\end{eqnarray}
Finally, we note that \eqref{scenario-1-infeasibility}, \eqref{scenario-1-obj-diff} imply \eqref{scenario-1-ergodic-N}.
\end{proof}

\section{Iteration Complexity of ADMM: Kurdyka-{\L}ojasiewicz Property}\label{sec:scenario-2}
In this section, we prove an $O(1/\epsilon)$ iteration complexity of ADMM \eqref{admm-N} under the conditions in Scenario~2 of Table \ref{tab:N-scenarios}. Indeed, we prove that the ADMM for the original problem \eqref{prob:N} returns an $\epsilon$-optimal solution within $O(1/\epsilon)$ iterations in Scenario 2.

%For $i=2,\ldots,N-1$,
Under the conditions in Scenario 2, the multi-block ADMM \eqref{admm-N} for solving \eqref{prob:N} can be rewritten as: % (note that some constant terms in the subproblems are discarded):
\begin{eqnarray}
x_1^{k+1} &:=& \argmin_{x_1 \in \XCal_1} f_1(x_1) + \frac{\gamma}{2}\left\| A_1 x_1 + \sum\limits_{j=2}^{N-1} A_j x_j^k + x_N^k - b - \frac{1}{\gamma}\lambda^k \right\|^2, \label{scenario-2-update-x-1}\\
x_i^{k+1} &:=& \argmin_{x_i \in \XCal_i}f_i(x_i) + \frac{\gamma}{2}\left\| \sum\limits_{j=1}^{i-1} A_j x_j^{k+1} + A_i x_i + \sum\limits_{j=i+1}^{N-1} A_j x_j^k + x_N^k - b - \frac{1}{\gamma}\lambda^k \right\|^2, \nonumber \\
  & & \qquad \qquad \qquad \qquad \qquad \qquad \qquad \qquad \qquad \qquad \qquad \qquad
      i=2,\ldots,N-1, \label{scenario-2-update-x-i}\\
x_N^{k+1} &:=& \argmin f_N(x_N) +\frac{\gamma}{2}\left\| \sum\limits_{j=1}^{N-1} A_j x_j^{k+1} + x_N - b - \frac{1}{\gamma}\lambda^k \right\|^2, \label{scenario-2-update-x-N}\\
\lambda^{k+1} &:=& \lambda^k - \gamma \left(A_1 x_1^{k+1} + A_2 x_2^{k+1}+ \cdots + A_{N-1}x_{N-1}^{k+1} + x_N^{k+1} - b\right). \label{scenario-2-update-lambda}
\end{eqnarray}
The first-order optimality conditions for \eqref{scenario-2-update-x-1}-\eqref{scenario-2-update-x-N} are given respectively by $x_i^{k+1}\in\XCal_i, i=1,\ldots,N-1$, and
\begin{align}
& g_1(x_{1}^{k+1})-A_{1}^{\top}\lambda^{k}+\gamma A_{1}^{\top}\left(A_{1}x_{1}^{k+1}+\sum\limits_{j=2}^{N-1} A_{j}x_{j}^{k} + x_{N}^k -b\right) = 0, & \label{scenario-2-opt-x-1} \\
& g_i(x_{i}^{k+1})-A_{i}^{\top}\lambda^{k}+\gamma A_{i}^{\top}\left(\sum\limits_{j=1}^{i} A_{j}x_{j}^{k+1} + \sum\limits_{j=i+1}^{N-1} A_{j}x_{j}^{k} + x_N^k -b\right) = 0, & \label{scenario-2-opt-x-i} \\
& \nabla f_N(x_{N}^{k+1})-\lambda^{k}+\gamma \left(\sum\limits_{j=1}^{N-1} A_{j}x_{j}^{k+1} + x_{N}^{k+1}-b\right) = 0, & \label{scenario-2-opt-x-N}
\end{align}
where $g_i \in \partial \left( f_i + \mathbf{1}_{\XCal_i}\right)$ is a subgradient of $f_i + \mathbf{1}_{\XCal_i}$ for $i=1,2,\ldots,N-1$. Moreover, by combining with \eqref{scenario-2-update-lambda}, \eqref{scenario-2-opt-x-1}-\eqref{scenario-2-opt-x-N} can be rewritten as
\begin{align}
& g_{1}(x_{1}^{k+1})-A_{1}^{\top}\lambda^{k+1} +\gamma A_{1}^{\top}\left(\sum\limits_{j=2}^{N-1} A_{j}(x_{j}^{k}-x_{j}^{k+1})+(x_{N}^{k}-x_{N}^{k+1})\right) = 0, & \label{scenario-2-opt-x-1-lambda} \\
& g_i(x_{i}^{k+1})-A_{i}^{\top}\lambda^{k+1}+\gamma A_{i}^{\top}\left(\sum\limits_{j=i+1}^{N-1} A_{j}(x_{j}^{k}-x_{j}^{k+1})+(x_{N}^{k}-x_{N}^{k+1})\right) = 0, & \label{scenario-2-opt-x-i-lambda} \\
& \nabla f_N(x_{N}^{k+1})-\lambda^{k+1} = 0. &\label{scenario-2-opt-x-N-lambda}
\end{align}

%In order to make sure of the assumption
Note that in Scenario 2 we require that $\LCal_\gamma$ is a Kurdyka-{\L}ojasiewicz (KL) function. Let us first introduce the notion %definition
of the KL function and the KL property, which can be found, e.g., in
% . These notions and definitions follow from some standard literatures on KL functions (see, e.g.,
\cite{Bolte-characterizations-2010,Bolte-Sabach-Teboulle-2014}.
We denote $\dist(x,S):=\inf\{\| y-x \|:y\in S\}$ as the distance from $x$ to $S$. Let $\eta\in(0,+\infty]$. We further denote $\Phi_\eta$ to be the class of all concave and continuous functions $\varphi:[0,\eta)\rightarrow\RR_{+}$ satisfying the following conditions:
\begin{compactenum}
\item $\varphi(0)=0$;
\item $\varphi$ is $C^1$ on $(0,\eta)$ and continuous at $0$;
\item for all $s\in(0,\eta):\varphi'(s)>0$.
\end{compactenum}
\begin{definition}
Let $f:\Omega\rightarrow(-\infty,+\infty]$ be proper and lower semicontinuous.
\begin{compactenum}
\item The function $f$ has Kurdyka-{\L}ojasiewicz (KL) property at $w_0\in\{w\in\Omega:\partial f(w)\neq\emptyset\}$ if there exists $\eta\in(0,+\infty]$, a neighbourhood $W_0$ of $w_0$ and a function $\varphi\in\Phi_\eta$ such that for all
\begin{displaymath}
\bar{w}_0 \in W\cap\left\{w\in\Omega: f(w)<f(w_0)<f(w)+\eta\right\},
\end{displaymath}
the following inequality holds,
\begin{equation}\label{inequality-kl}
\varphi'(f(\bar{w}_0) - f(w_0))\dist(0,\partial f(\bar{w}_0))\geq 1.
\end{equation}
\item The function $f$ is a KL function if $f$ satisfies the KL property at each point of $\Omega\cap\{\partial f(w)\neq\emptyset\}$.
\end{compactenum}
\end{definition}

\begin{remark}
It is important to remark that most convex
functions from practical applications satisfy the KL property; see Section 5.1 of \cite{Bolte-Sabach-Teboulle-2014}. In fact, convex functions that do not satisfy the KL property exist (see \cite{Bolte-characterizations-2010} for a counterexample) but they are rare and difficult to construct.
%A general convex function may not satisfy the KL property (see \cite{Bolte-characterizations-2010} for a counterexample). However, in many applications $\LCal_\gamma$ is KL function.
Indeed, $\LCal_\gamma$ will be a KL function if each $f_i$ satisfies growth condition, or uniform convexity, or they are general convex semialgebraic or real analytic functions. We refer the interested readers to \cite{Attouch-Bolte-Svaiter-2010} and \cite{Bolte-Sabach-Teboulle-2014} for more information.
\end{remark}

The following result, which is called {\em uniformized KL property}, is from Lemma 6 of \cite{Bolte-Sabach-Teboulle-2014}.
\begin{lemma}\label{lemma_KL_property}[Lemma 6 \cite{Bolte-Sabach-Teboulle-2014}]
Let $\Omega$ be a compact set and $f: \RR^n\rightarrow(-\infty,\infty]$ be a proper and lower semi-continuous function. Assume that $f$ is constant on $\Omega$ and satisfies the KL property at each point of $\Omega$. Then, there exists $\epsilon>0$, $\eta>0$ and $\varphi\in\Phi_\eta$ such that for all $\bar{u}$ in $\Omega$ and all $u$ in the intersection:
\begin{displaymath}
\left\{u\in\RR^n: \dist(u, \Omega) < \epsilon\right\}\cap\left\{u\in\RR^n: f(\bar{u}) < f(u) < f(\bar{u})+\eta\right\},
\end{displaymath}
the following inequality holds,
\begin{displaymath}
\varphi'\left( f(u) - f(\bar{u})\right)\dist\left(0, \partial f(u)\right)\geq 1.
\end{displaymath}
\end{lemma}

%\begin{proof}
%The proof is omitted for succinctness. We refer the readers to Lemma 6 of \cite{Bolte-Sabach-Teboulle-2014} for details.
%\end{proof}

We now give a formal definition of the limit point set. Let the sequence $w^k = \left(x_1^k,\ldots, x_N^k,\lambda^k\right)$ be a sequence generated by the multi-ADMM \eqref{admm-N} from a starting point $w^0 = \left(x_1^0,\ldots, x_N^0,\lambda^0\right)$. The set of all limit points is denoted by $\Omega(w^0)$, i.e.,
%\begin{eqnarray*}
%\Omega(w^0) = & & \left\{ \bar{w}\in\RR^{n_1}\times\ldots\times\RR^{n_N}\times\RR^p: \exists\text{ an increasing sequence of integers }\{k_l\}_{l=1,\ldots}\right. \\
%& & \left.\text{ such that }w^{k_l}\rightarrow\bar{w}\text{ as }l\rightarrow\infty \right\}.
%\end{eqnarray*}
\[
\Omega(w^0) = \left\{ \bar{w}\in\RR^{n_1}\times\cdots\times\RR^{n_N}\times\RR^p: \exists
\text{ an infinite sequence }\{k_l\}_{l=1,\ldots}
\text{ such that }w^{k_l}\rightarrow\bar{w}\text{ as }l\rightarrow\infty \right\}.
\]

In the following we present the main results in this section. Specifically, Theorem \ref{thm-l2-convergence} gives the convergence of the multi-ADMM \eqref{admm-N}, and we include its proof in the Appendix. Theorem \ref{thm-l1-convergence} shows that the whole sequence generated by the multi-ADMM \eqref{admm-N} converges.

\begin{theorem}\label{thm-l2-convergence}
Under the conditions in Scenario 2 of Table \ref{tab:N-scenarios}, %{\color{blue} assuming $\Omega(w^0)$ is non-empty},
then: %we have the following conclusions:
%The following conclusions hold true in Scenario 2.
\begin{enumerate}
\item $\Omega(w^0)$ is a non-empty set, and any point in $\Omega(w^0)$ is a stationary point of $\LCal_\gamma(x_1,\ldots,x_N,\lambda)$;
\item $\Omega(w^0)$ is a compact and connected set;
\item The function $\LCal_\gamma(x_1,\ldots,x_N,\lambda)$ is finite and constant on  $\Omega(w^0)$.
\end{enumerate}
\end{theorem}

%{\color{red} (We should be able to prove optimality in the convex case.)}

\begin{remark}
%{\color{blue} {\ }
In Theorem~\ref{thm-l2-convergence}, we do not require $\LCal_\gamma$ to be a KL function, which is only required in Theorem~\ref{thm-l1-convergence} (see next).
%\begin{enumerate}
%\item In Theorem~\ref{thm-l2-convergence}, we do not require $\LCal_\gamma$ to be a KL function, which is only required in Theorem~\ref{thm-l1-convergence} (see next).
%\item
%It should be pointed out that $\Omega(w^0)$ might be empty. However, if $\XCal_i$ is assumed to be compact for $i=1,2,\ldots,N-1$, then $\Omega(w^0)$ is non-empty. Specifically, the boundedness of $\left(x_1^k,\ldots,x_{N-1}^k\right)$ is obtained from the compactness of $\XCal_i$ for $i=1,\ldots,N-1$. The boundedness of $\left( x_N^k, \lambda^k\right)$ can be proved by using \eqref{lambda-bound-x} and \eqref{inequality-x-lambda-limit} in the Appendix. Finally, we conclude that $\Omega(w^0)$ is non-empty by using the Weierstrass Theorem.
%\end{enumerate}
%}
\end{remark}

\begin{theorem}\label{thm-l1-convergence}
Suppose that $\LCal_\gamma(x_1,\ldots, x_N,\lambda)$ is a KL function. Let the sequence $w^k =\left(x_1^k,\ldots, x_N^k,\lambda^k\right)$ be generated by the multi-block ADMM \eqref{admm-N}. Let $w^* =\left(x_1^*,\ldots, x_N^*,\lambda^*\right)\in\Omega(w^0)$, the sequence $w^k =\left(x_1^k,\ldots, x_N^k,\lambda^k\right)$ has a finite length, i.e.,
\begin{eqnarray}\label{inequality-convergence-l1}
\sum\limits_{k=0}^\infty \left(\sum\limits_{i=1}^{N-1} \| A_i x_i^k - A_i x_i^{k+1}\| + \| x_N^k - x_N^{k+1} \|+ \| \lambda^k - \lambda^{k+1}\| \right) \leq G,
\end{eqnarray}
where the constant $G$ is given by
\begin{displaymath}
G:= 2\left(\sum\limits_{i=1}^{N-1} \| A_i x_i^0 - A_i x_i^1\| + \| x_N^0 - x_N^1 \|+ \| \lambda^0 - \lambda^1\| \right) + \frac{2M\gamma(1+L^2)}{\gamma^2-2L^2}\varphi\left(\LCal_\gamma(w^1) - \LCal_\gamma(w^*)\right),
\end{displaymath}
and
\begin{displaymath}
M = \max\left( \gamma\sum\limits_{i=1}^{N-1}\left\| A_i^\top \right\|, \frac{1}{\gamma}+1+\sum\limits_{i=1}^{N-1} \left\| A_i^\top \right\|\right)>0,
\end{displaymath}
and the whole sequence $\left(A_1 x_1^k, A_2 x_2^k, \ldots, A_{N-1}x_{N-1}^k, x_N^k,\lambda^k\right)$ converges to $\left(A_1 x_1^*,\ldots, A_{N-1}x_{N-1}^*, x_N^*,\lambda^*\right)$.
\end{theorem}

\begin{proof}
The proof of this theorem is almost identical to the proof of Theorem 1 in \cite{Bolte-Sabach-Teboulle-2014}, by utilizing the uniformized KL property (Lemma \ref{lemma_KL_property}), and the facts that $\Omega(w^0)$ is compact, $\LCal_\gamma(w)$ is constant (proved in Theorem \ref{thm-l2-convergence}), with function $\Psi$ replaced by $\LCal_\gamma$ and some other minor changes. We thus omit the proof for succinctness.
\end{proof}

Based on Theorem \ref{thm-l1-convergence}, we prove a key lemma for analyzing the iteration complexity for the ADMM.

\begin{lemma}\label{lemma-kurdyka-lojasiewicz-property}
Let $(x_1^{k+1},x_{2}^{k+1},\ldots,x_N^{k+1},\lambda^{k+1})\in\Omega$ be generated by the multi-ADMM \eqref{scenario-2-update-x-1}-\eqref{scenario-2-update-lambda} (or equivalently, \eqref{admm-N}) from given $(x_{2}^{k},\ldots,x_{N}^{k},\lambda^{k})$. For any $u^*=(x_1^*,x_2^*,\ldots,x_N^*)\in\Omega^*$ and $\lambda\in\RR^p$, it holds in Scenario 2 that
\begin{eqnarray}
& & f(u^*)-f(u^{k+1})+\left(\begin{array}{c} x_1^* - x_1^{k+1} \\ x_2^* - x_2^{k+1} \\ \vdots \\ x_N^* - x_N^{k+1} \\ \lambda-\lambda^{k+1}\end{array} \right)^{\top}
\left(\begin{array}{c} -A_1^\top\lambda^{k+1} \\ -A_2^\top\lambda^k+1 \\ \vdots \\ -\lambda^{k+1} \\  \sum_{i=1}^{N-1} A_i x_i^{k+1} + x_N^{k+1} -b \end{array} \right) \nonumber \\
& & + \frac{\gamma}{2}\left(\left\| A_{1}x_{1}^* + \sum\limits_{i=2}^{N-1} A_i x_i^{k} + x_N^k - b\right\|^2 - \left\| A_1 x_1^* + \sum\limits_{i=2}^{N-1} A_i x_i^{k+1} + x_N^{k+1} - b\right\|^2 \right) \nonumber \\
& & + \frac{1}{2\gamma}\left(\left\|\lambda-\lambda^k \right\|^2 - \left\|\lambda-\lambda^{k+1} \right\|^2 \right) + \gamma D(N-2)\left( \sum\limits_{i=1}^{N-1} \left\| A_i x_i^k - A_i x_i^{k+1}\right\| + \left\| x_N^k - x_N^{k+1}\right\| \right) \nonumber \\
& \geq & 0, \label{inequality-kurdyka-lojasiewicz-property}
\end{eqnarray}
where $D$ is a constant.
\end{lemma}

\begin{proof}
Note that combining %\eqref{scenario-2-opt-x-1-lambda}
\eqref{scenario-2-opt-x-i-lambda}-\eqref{scenario-2-opt-x-N-lambda} yields
\begin{eqnarray}\label{inequality-kurdyka-lojasiewicz-property-1}
& & \left(\begin{array}{c} x_1 - x_1^{k+1} \\ x_2 - x_2^{k+1} \\ \vdots \\ x_N - x_N^{k+1} \end{array} \right)^\top\left[
\left(\begin{array}{c} g_1(x_1^{k+1}) - A_1^\top\lambda^{k+1} \\ g_2(x_2^{k+1}) - A_2^\top\lambda^{k+1} \\ \vdots \\ \nabla f_N(x_N^{k+1}) - \lambda^{k+1} \end{array} \right) + \left(\begin{array}{cccc} \gamma A_1^\top A_2 & \gamma A_1^\top A_3 & \cdots & \gamma A_1^\top \\ 0 & \gamma A_2^\top A_3 & \cdots & \gamma A_2^\top \\ \vdots & \ddots & \ddots & \vdots \\ 0 & 0 & \cdots & \gamma A_{N-1}^\top \\ 0 & 0 & \cdots & 0 \end{array} \right)\left(\begin{array}{c} x_2^k - x_2^{k+1} \\ \vdots \\ x_N^k - x_N^{k+1} \end{array} \right)\right] \nonumber \\
& \ge  & 0,
\end{eqnarray}
where $x_i\in\XCal_i$ and $g_i\in\partial(f_i + \mathbf{1}_{\XCal_i})$ is a subgradient of $f_i + \mathbf{1}_{\XCal_i}$ for $i=1,2,\ldots,N-1$.

The key step in our proof is to bound the following terms
\begin{displaymath}
\left( x_i - x_i^{k+1}\right)^\top A_i^\top \left( \sum\limits_{j=i+1}^{N-1} A_j (x_j^k - x_j^{k+1}) + (x_N^k - x_N^{k+1})\right), \ i=1,2,\ldots, N-1.
\end{displaymath}
For the first term, we have (similar to Lemma \ref{lemma-associated-perturbation})
\begin{eqnarray*}
& & (x_1 - x_1^{k+1})^\top A_1^\top\left[ \sum\limits_{j=2}^{N-1} A_j (x_j^k - x_j^{k+1}) + (x_N^k - x_N^{k+1}) \right] \\
& \leq & \frac{1}{2}\left(\left\| A_1 x_1 +\sum\limits_{j=2}^{N-1} A_j x_j^k + x_N^k -b\right\|^2 - \left\| A_1 x_1 + \sum\limits_{j=2}^{N-1} A_j x_j^{k+1} + x_N^{k+1} - b\right\|^2 \right) + \frac{1}{2\gamma^2}\|\lambda^{k+1}-\lambda^k\|^2.
\end{eqnarray*}

For $i=2,3,\ldots,N-1$, we have,
\begin{eqnarray*}
& & (x_i - x_i^{k+1})^\top A_i^\top\left[ \sum\limits_{j=i+1}^{N-1} A_j (x_j^k - x_j^{k+1}) + (x_N^k - x_N^{k+1})\right] \\
& \leq & \left\| A_i x_i - A_i x_i^{k+1}\right\|\left[ \sum\limits_{j=i+1}^{N-1} \left\| A_j x_j^k - A_j x_j^{k+1}\right\| + \left\| x_N^k - x_N^{k+1}\right\| \right] \\
& \leq & \left\| A_i x_i - A_i x_i^{k+1}\right\|\left[ \sum\limits_{j=1}^{N-1} \left\| A_j x_j^k - A_j x_j^{k+1}\right\| + \left\| x_N^k - x_N^{k+1}\right\| \right].
\end{eqnarray*}

Therefore, %we have
\begin{eqnarray}
& &  \left(\begin{array}{c} x_1 - x_1^{k+1} \\ x_2 - x_2^{k+1} \\ \vdots \\ x_N - x_N^{k+1} \end{array} \right)^{\top} \left(\begin{array}{cccc} \gamma A_1^\top A_2 & \gamma A_1^\top A_3 & \cdots & \gamma A_1^\top \\ 0 & \gamma A_2^\top A_3 & \cdots & \gamma A_2^\top \\ \vdots & \ddots & \ddots & \vdots \\ 0 & 0 & \cdots & \gamma A_{N-1}^\top \\ 0 & 0 & \cdots & 0 \end{array} \right)\left(\begin{array}{c} x_2^k - x_2^{k+1} \\ \vdots \\ x_N^k - x_N^{k+1} \end{array} \right) \nonumber \\
& \leq & \frac{\gamma}{2}\left(\left\| A_1 x_1 + \sum\limits_{i=2}^{N-1} A_i x_i^{k} +x_N^k- b\right\|^2 - \left\| A_1 x_1 + \sum\limits_{i=2}^{N-1} A_i x_i^{k+1} + x_N^{k+1} - b\right\|^2 \right) + \frac{1}{2\gamma}\left\|\lambda^{k+1}-\lambda^{k}\right\|^2 \nonumber \\
 & & + \gamma \left(\sum\limits_{i=2}^{N-1}\left\| A_i x_i - A_i x_i^{k+1}\right\|\right)\left[ \sum\limits_{i=1}^{N-1} \left\| A_i x_i^k - A_i x_i^{k+1}\right\| + \left\| x_N^k - x_N^{k+1}\right\| \right] . \label{inequality-kurdyka-lojasiewicz-property-2}
\end{eqnarray}
Combining \eqref{scenario-2-update-lambda}, \eqref{inequality-kurdyka-lojasiewicz-property-1} and \eqref{inequality-kurdyka-lojasiewicz-property-2}, it holds for any $\lambda\in\RR^p$ that
\begin{eqnarray}
& & \left(\begin{array}{c} x_1 - x_1^{k+1} \\ x_2 - x_2^{k+1} \\ \vdots \\ x_N - x_N^{k+1} \\ \lambda-\lambda^{k+1}\end{array} \right)^{\top}
\left(\begin{array}{c} g_1(x_1^{k+1}) - A_1^\top\lambda^{k+1} \\ g_2(x_2^{k+1}) - A_2^\top\lambda^{k+1} \\ \vdots \\ \nabla f_N(x_N^{k+1}) - \lambda^{k+1} \\ \sum_{i=1}^{N-1} A_i x_i^{k+1} + x_N^{k+1} -b \end{array} \right) + \frac{1}{\gamma}\left(\lambda-\lambda^{k+1}\right)^\top\left(\lambda^{k+1} - \lambda^k\right) \nonumber \\
& & + \frac{\gamma}{2}\left(\left\| A_1 x_1 + \sum\limits_{i=2}^{N-1} A_i x_i^k + x_N^k - b\right\|^2 - \left\| A_1 x_1 + \sum\limits_{i=2}^{N-1} A_i x_i^{k+1} + x_N^{k+1} - b\right\|^2\right) + \frac{1}{2\gamma}\left\|\lambda^{k+1}-\lambda^k \right\|^2 \nonumber \\
& & + \gamma \left(\sum\limits_{i=2}^{N-1}\left\| A_i x_i - A_i x_i^{k+1}\right\|\right)\left[ \sum\limits_{i=1}^{N-1} \left\| A_i x_i^k - A_i x_i^{k+1}\right\| + \left\| x_N^k - x_N^{k+1}\right\| \right] \nonumber \\
& \geq & 0. \label{inequality-kurdyka-lojasiewicz-property-3}
\end{eqnarray}
Using \eqref{identity-3}, we have
\begin{eqnarray*}
\frac{1}{\gamma}\left(\lambda-\lambda^{k+1}\right)^{\top}\left(\lambda^{k+1}-\lambda^{k}\right)+\frac{1}{2\gamma}\left\|\lambda^{k+1}-\lambda^{k}\right\|^2 = \frac{1}{2\gamma}\left(\left\|\lambda-\lambda^k \right\|^2 - \left\|\lambda-\lambda^{k+1}\right\|^2 \right).
\end{eqnarray*}
Letting $u=u^{*}$ in \eqref{inequality-kurdyka-lojasiewicz-property-3}, and invoking the convexity of $f_i$, we obtain
\begin{eqnarray*}
& & f(u^*)-f(u^{k+1})+\left(\begin{array}{c} x_1^* - x_1^{k+1} \\ x_2^* - x_2^{k+1} \\ \vdots \\ x_N^* - x_N^{k+1} \\ \lambda-\lambda^{k+1}\end{array} \right)^\top
\left(\begin{array}{c} -A_1^\top\lambda^{k+1} \\ -A_2^\top\lambda^{k+1} \\ \vdots \\ -\lambda^{k+1} \\  \sum_{i=1}^{N-1} A_i x_i^{k+1} + x_N^{k+1} - b \end{array} \right) + \frac{1}{2\gamma}\left(\left\|\lambda-\lambda^k \right\|^2 - \left\|\lambda-\lambda^{k+1} \right\|^2 \right) \\
& & + \frac{\gamma}{2}\left(\left\| A_{1}x_{1}^* + \sum\limits_{i=2}^{N-1} A_i x_i^k + x_N^k - b\right\|^2 - \left\| A_1 x_1^* + \sum\limits_{i=2}^{N-1} A_i x_i^{k+1} + x_N^{k+1} - b\right\|^2 \right)  \\
& & + \gamma \left(\sum\limits_{i=2}^{N-1}\left\| A_i x_i^* - A_i x_i^{k+1}\right\|\right)\left[ \sum\limits_{i=1}^{N-1} \left\| A_i x_i^k - A_i x_i^{k+1}\right\| + \left\| x_N^k - x_N^{k+1}\right\| \right] \\
& \geq & 0.
\end{eqnarray*}
From Theorem \ref{thm-l1-convergence} we know that the whole sequence $\left(A_1 x_1^k, A_2 x_2^k, \ldots, A_{N-1}x_{N-1}^k, x_N^k,\lambda^k\right)$ converges to $\left(A_1 x_1^*,\ldots, A_{N-1}x_{N-1}^*, x_N^*,\lambda^*\right)$. Therefore, there exists a constant $D>0$ such that
\begin{eqnarray}
\left\| A_i x_i^k - A_i x_i^{k+1}\right\| \leq D,
\end{eqnarray}
for any $k\geq 0$ and any $i=2,3,\ldots,N-1$. This implies \eqref{inequality-kurdyka-lojasiewicz-property}.
\end{proof}

Now, we are ready to prove the $O(1/\epsilon)$ iteration complexity of the multi-block ADMM for \eqref{prob:N}.
\begin{theorem}\label{thm-scenario-2-ergodic-N}
Let $(x_1^{k+1},x_2^{k+1},\ldots,x_N^{k+1},\lambda^{k+1})\in\Omega$ be generated by ADMM \eqref{scenario-2-update-x-1}-\eqref{scenario-2-update-lambda} from given $(x_2^k,\ldots,x_N^k,\lambda^k)$. For any integer $t>0$, let $\bar{u}^t = (\bar{x}_1^t, \bar{x}_2^t, \ldots, \bar{x}_N^t)$ and $\bar{\lambda}^{t}$ be defined as
\begin{eqnarray*}
\bar{x}_i^t = \frac{1}{t+1}\sum\limits_{k=0}^{t}x_i^{k+1}, \ i=1,2,\ldots,N, \quad
\bar{\lambda}^t = \frac{1}{t+1}\sum\limits_{k=0}^t \lambda^{k+1}.
\end{eqnarray*}
For any $(u^*,\lambda^*) \in\Omega^{*}$, by defining $\rho:=\|\lambda^*\|+1$, it holds in Scenario 2 that,
\begin{eqnarray*}
 0 & \leq & f(\bar{u}^t)-f(u^*)+\rho\left\| \sum\limits_{i=1}^N A_i\bar{x}_i^t-b \right\| \\
& \leq & \frac{\rho^2 + \|\lambda^0\|^2}{\gamma (t+1)} + \frac{\gamma}{2(t+1)}\left\| \sum\limits_{i=2}^{N-1} A_i (x_i^0 - x_j^*) + (x_N^0 - x_N^*)\right\|^2 + \frac{\gamma DG}{t+1}.
\end{eqnarray*}
Note this also implies that when $t=O(1/\epsilon)$, $\bar{u}^{t}=(\bar{x}_{1}^{t}, \bar{x}_{2}^{t}, \ldots, \bar{x}_{N}^{t})$ is an $\epsilon$-optimal solution to the original problem \eqref{prob:N}, i.e., both the error of the objective function value and the residual of the equality constraint satisfy that
\be\label{scenario-2-ergodic-N}|f(\bar{u}^t)-f(u^*)| = O(\epsilon), \quad \mbox{ and } \quad \left\| \sum\limits_{i=1}^N A_i\bar{x}_i^t-b \right\|=O(\epsilon).\ee
\end{theorem}
\begin{proof}
Because $(u^{k},\lambda^k)\in\Omega$, it holds that $(\bar{u}^{t},\bar{\lambda}^t)\in\Omega$ for all $t\geq 0$.
By Lemma \ref{lemma-kurdyka-lojasiewicz-property} and invoking the convexity of function $f(\cdot)$, we have
\begingroup
\allowdisplaybreaks
\begin{align*}
 & f(u^*)-f(\bar{u}^t) + \lambda^\top \left( \sum\limits_{i=1}^{N-1} A_i\bar{x}_i^t+\bar{x}_N^t-b \right) \\
 = & f(u^*)-f(\bar{u}^t)+\left(\begin{array}{c} x_1^* - \bar{x}_1^t \\ x_2^* - \bar{x}_2^t \\ \vdots \\ x_N^* -\bar{x}_N^t \\ \lambda-\bar{\lambda}^t\end{array} \right)^\top
\left(\begin{array}{c} -A_1^\top\bar{\lambda}^t \\ -A_2^\top\bar{\lambda}^t \\ \vdots \\ -\bar{\lambda}^t \\ \sum_{i=1}^{N-1} A_i\bar{x}_i^t + \bar{x}_N^t - b \end{array} \right)\\
 \geq & \frac{1}{t+1}\sum\limits_{k=0}^t\left[f(u^*)-f(u^{k+1})+\left(\begin{array}{c} x_1^* - x_1^{k+1} \\ x_2^* - x_2^{k+1} \\ \vdots \\ x_N^* - x_N^{k+1} \\ \lambda-\lambda^{k+1}\end{array} \right)^\top
\left(\begin{array}{c} -A_1^\top\lambda^{k+1} \\ -A_2^\top\lambda^{k+1} \\ \vdots \\ -\lambda^{k+1} \\ \sum_{i=1}^{N-1} A_i x_i^{k+1} + x_N^{k+1} - b \end{array} \right)\right]\\
 \geq & \frac{1}{t+1}\sum\limits_{k=0}^t\left[ \frac{1}{2\gamma}\left(\left\|\lambda-\lambda^{k+1}\right\|^2 - \left\|\lambda-\lambda^k \right\|^2 \right) - \gamma D(N-2)\left( \sum\limits_{i=1}^{N-1} \left\| A_i x_i^k - A_i x_i^{k+1}\right\| + \left\| x_N^k - x_N^{k+1}\right\| \right) \right. \\
 & \left. +\frac{\gamma}{2}\left(\left\| A_1 x_1^* + \sum\limits_{i=2}^{N-1} A_i x_i^{k+1} + x_N^{k+1} - b\right\|^2 - \left\| A_1 x_1^* + \sum\limits_{i=2}^{N-1} A_i x_i^k + x_N^k - b\right\|^2 \right)\right]  \\
 \geq & -\frac{1}{2\gamma(t+1)}\left\|\lambda-\lambda^0 \right\|^2 - \frac{\gamma}{2(t+1)}\left\| A_1 x_1^* + \sum\limits_{i=2}^{N-1} A_i x_i^0 + x_N^0 - b\right\|^2 \\
 & - \frac{\gamma D (N-2)}{t+1}\sum\limits_{k=0}^t \left( \sum\limits_{i=1}^{N-1} \left\| A_i x_i^k - A_i x_i^{k+1}\right\| + \left\| x_N^k - x_N^{k+1}\right\| \right) \\
 \geq & -\frac{1}{2\gamma(t+1)}\left\|\lambda-\lambda^0\right\|^2 - \frac{\gamma}{2(t+1)}\left\| A_1 x_1^* + \sum\limits_{i=2}^{N-1} A_i x_i^0 + x_N^0 - b\right\|^2 - \frac{\gamma D G (N-2)}{t+1},
\end{align*}
\endgroup
where the last inequality holds due to Theorem \ref{thm-l1-convergence}.
Note that this inequality holds for all $\lambda\in\RR^{p}$. From the optimal condition \eqref{kkt} we obtain
\[0 \geq f(u^*)-f(\bar{u}^t) + (\lambda^*)^\top \left( \sum\limits_{i=1}^{N-1} A_i\bar{x}_i^t+\bar{x}_N^t-b \right).\]
Moreover, since $\rho:=\|\lambda^*\|+1$,
$\| \lambda -\lambda_0\|^2 \le 2(\rho^{2}+\|\lambda^0\|^2)$ for all $\|\lambda\|\leq\rho$, and $\sum_{i=1}^{N-1} A_i x^*_i + x^*_N =b$, we obtain
\begin{eqnarray}\label{scenario-2-ergodic-N-inequality-1}
0 & \leq & f(\bar{u}^t)-f(u^*)+\rho\left\| \sum\limits_{i=1}^{N-1} A_i\bar{x}_i^t+\bar{x}_N^t-b \right\| \nonumber \\
& \leq & \frac{\rho^2 + \|\lambda^0\|^2}{\gamma (t+1)} + \frac{\gamma}{2(t+1)}\left\| \sum\limits_{i=2}^{N-1} A_i (x_i^0 - x_i^*) + (x_N^0 - x_N^*)\right\|^{2} + \frac{\gamma D G (N-2)}{t+1}.
\end{eqnarray}
When $t=O(1/\epsilon)$, we have
\begin{eqnarray}\label{scenario-2-error-order}
\frac{\rho^2 + \|\lambda^0\|^2}{\gamma (t+1)} + \frac{\gamma}{2(t+1)}\left\| \sum\limits_{i=2}^{N-1} A_i (x_i^0 - x_i^*) + (x_N^0 - x_N^*)\right\|^2 + \frac{\gamma D G (N-2)}{t+1} = O(\epsilon).
\end{eqnarray}

By the same argument as in the proof for Theorem~\ref{thm-scenario-1-ergodic-N}, \eqref{scenario-2-ergodic-N} follows from \eqref{scenario-2-error-order}.

\end{proof}

%\section{Conclusions}\label{sec:conclusion}

%In this paper, we analyzed the sublinear convergence rate of the multi-block ADMM under two scenarios. These are the first sublinear convergence rate results for multi-block ADMM without the strongly convexity of any objective function.

%{\color{red} (The Conclusion section appears to be quite pale. We may as well just skip that part.)}

\section*{Acknowledgements}

Research of Shiqian Ma was supported in part by the Hong Kong Research Grants Council
General Research Fund Early Career Scheme (Project ID: CUHK 439513). Research of Shuzhong Zhang was supported in part by the National Science Foundation under Grant Number CMMI-1161242.

%{\color{red} (Check the consistency of the references; e.g. the style of initials in [8])}.

\bibliographystyle{plain}
\bibliography{admmN}

\appendix
\section{Proof of Theorem \ref{thm-l2-convergence}}
We first prove a key lemma in the proof of Theorem \ref{thm-l2-convergence}.
\begin{lemma} \label{LemmaA1}
The following holds in Scenario 2,
\begin{enumerate}
\item The iterative gap of dual variable can be bounded by that of primal variable, i.e.,
\begin{equation}\label{opt-cond-N-1}
\nabla f_N(x_N^{k+1}) = \lambda^{k+1},
\end{equation}
and
\begin{equation}\label{lambda-bound-x}
\left\| \lambda^{k+1} - \lambda^{k} \right\|^2 \leq L^2 \left\| x_N^{k+1} - x_N^k \right\|,
\end{equation}
where $L$ satisfies that
\begin{displaymath}
\left\| \nabla f_N(x) - \nabla f_N(y) \right\| \leq L\left\| x - y \right\|.
\end{displaymath}
\item  The augmented Lagrangian $L_\gamma$ has a sufficient decrease in each iteration, i.e.,
\begin{eqnarray}\label{L-sufficient-decrease}
& & \LCal_\gamma(x_1^{k},\ldots, x_{N+1}^{k},\lambda^k) - \LCal_\gamma(x_1^{k+1},\ldots, x_{N+1}^{k+1},\lambda^{k+1}) \nonumber \\
& \geq & \frac{\gamma^2-2L^2}{2\gamma(1+L^2)}\left(\sum\limits_{i=1}^{N-1} \left\| A_i x_i^k - A_i x_i^{k+1} \right\|^2 + \left\| x_N^k - x_N^{k+1} \right\|^2 + \left\| \lambda^k - \lambda^{k+1} \right\|^2\right).
\end{eqnarray}
\item The augmented Lagrangian $\LCal_\gamma(w^k)$ is uniformly lower bounded, and it holds true that
\begin{equation}\label{inequality-x-lambda}
\sum\limits_{k=0}^\infty \left(\sum\limits_{i=1}^{N-1} \left\| A_i x_i^{k+1} - A_i x_i^k \right\|^2 + \left\| x_N^{k+1} - x_N^k \right\|^2 + \left\|\lambda^{k+1}-\lambda^k \right\|^2\right) \leq \frac{2\gamma(1+L^2)}{\gamma^2-2L^2}\left( \LCal_\gamma(w^0) - L^*\right)
\end{equation}
where $L^*$ is the uniformly lower bound of $\LCal_\gamma(w^k)$, and hence
\begin{equation}\label{inequality-x-lambda-limit}
\lim\limits_{k\rightarrow\infty} \left( \sum_{i=1}^{N-1} \left\| A_i x_i^k - A_i x_i^{k+1} \right\|^2 + \left\| x_N^k - x_N^{k+1} \right\|^2 + \left\| \lambda^k - \lambda^{k+1} \right\|^2 \right)= 0.
\end{equation}
Moreover, $\left\{\left(x_1^k, x_2^k, \ldots,x_N^k, \lambda^k\right): k=0,1,\ldots\right\}$ is a bounded sequence.
\item There exists a upper bound for a subgradient of augmented Lagrangian $\LCal_\gamma$ in each iteration. Indeed, we define
\begin{eqnarray*}
R_i^{k+1} = \gamma A_i^\top \left(\sum\limits_{i=1}^{N-1} A_i x_i^{k+1} + x_N^{k+1} - b\right) - \gamma A_{i}^{\top}\left(\sum\limits_{j=i+1}^{N-1} A_{j}(x_{j}^{k}-x_{j}^{k+1})+(x_{N}^{k}-x_{N}^{k+1})\right)
\end{eqnarray*}
and
\[
R_N^{k+1} = \gamma \left(\sum\limits_{i=1}^{N-1} A_i x_i^{k+1} + x_N^{k+1} - b\right), \quad R_{\lambda}^{k+1} = b - \sum\limits_{i=1}^{N-1} A_i x_i^{k+1} - x_N^{k+1}
\]
for each positive integer $k$, and $i = 1,2,\ldots, N$. Then $\left( R_1^{k+1}, \ldots, R_N^{k+1}, R_\lambda^{k+1}\right)\in\partial\LCal_\gamma(w^{k+1})$. Moreover, it holds that
\begin{eqnarray}\label{inequality-subgradient-upper-bound}
& & \left\| \left( R_1^{k+1}, \ldots, R_N^{k+1}, R_\lambda^{k+1}\right) \right\| \nonumber \\
& \leq & \sum\limits_{i=1}^N \left\| R_i^{k+1} \right\| + \left\| R_\lambda^{k+1} \right\| \nonumber \\
& \leq & M\left( \sum\limits_{i=1}^{N-1} \left\| A_i x_i^k - A_i x_i^{k+1}\right\| + \left\| x_i^k - x_i^{k+1} \right\| + \left\| \lambda^k - \lambda^{k+1} \right\|\right), \quad \forall k\geq 0,
\end{eqnarray}
where $M$ is a constant defined as
\begin{equation}\label{def-M}
M = \max\left( \gamma\sum\limits_{i=1}^{N-1}\left\| A_i^\top \right\|, \frac{1}{\gamma}+1+\sum\limits_{i=1}^{N-1} \left\| A_i^\top \right\|\right)>0.
\end{equation}
\end{enumerate}
\end{lemma}
%\begin{proof}

\textbf{Proof of Lemma~\ref{LemmaA1}.}

\begin{enumerate}
\item \eqref{opt-cond-N-1} follows from \eqref{scenario-2-opt-x-N-lambda} directly. Then we consider the inequality \eqref{lambda-bound-x}. It follows from \eqref{opt-cond-N-1} and the fact that $\nabla f_N$ is Lipschitz continuous with $L$ that
\begin{eqnarray*}
\left\| \lambda^{k+1} - \lambda^k \right\|^2
= \left\| \nabla f(x_N^{k+1}) - \nabla f(x_N^k) \right\|^2 \leq L^2 \left\| x_N^{k+1} - x_N^k \right\|^2.
\end{eqnarray*}
\item Multiply both sides of \eqref{scenario-2-opt-x-1} by $x_1^k - x_1^{k+1}$, and invoking the convexity of $f_1$, we have
\begingroup
\allowdisplaybreaks
\begin{align}\label{inequality-decrease-1}
0 = & \left(x_1^k - x_1^{k+1}\right)^\top\left[ g_1(x_{1}^{k+1})-A_{1}^{\top}\lambda^{k}+\gamma A_{1}^{\top}\left(A_{1}x_{1}^{k+1}+\sum\limits_{j=2}^{N-1} A_{j}x_{j}^{k} + x_{N}^k -b\right) \right] \nonumber \\
 \leq & f(x_1^k) - f(x_1^{k+1}) - \left(A_1 x_1^k - A_1 x_1^{k+1}\right)^\top\lambda^k  \nonumber \\
 & + \gamma\left(A_1 x_1^k - A_1 x_1^{k+1}\right)^\top\left(A_{1}x_{1}^{k+1}+\sum\limits_{j=2}^{N-1} A_{j}x_{j}^{k} + x_{N}^k -b\right) \nonumber \\
 = & \left( f(x_1^k) - A_1 x_1^k + \frac{\gamma}{2}\left\| \sum\limits_{j=1}^{N-1} A_{j}x_{j}^{k} + x_{N}^k -b\right\|^2\right) - \frac{\gamma}{2}\left\| A_1 x_1^k - A_1 x_1^{k+1}\right\|^2 \nonumber \\
 & - \left( f(x_1^{k+1}) - A_1 x_1^{k+1} + \frac{\gamma}{2}\left\| A_1 x_1^{k+1} + \sum\limits_{j=2}^{N-1} A_{j}x_{j}^{k} + x_{N}^k -b\right\|^2 \right) \nonumber \\
 = & \LCal_\gamma(x_1^k,\ldots, x_N^k,\lambda^k) - \LCal_\gamma(x_1^{k+1},x_2^k,\ldots, x_N^k,\lambda^k) - \frac{\gamma}{2}\left\| A_1 x_1^k - A_1 x_1^{k+1}\right\|^2
\end{align}
\endgroup
where the second equality holds due to \eqref{identity-3}.

For $i=2,3,\ldots,N$, we can derive from \eqref{scenario-2-opt-x-i} and \eqref{scenario-2-opt-x-N} that %, and get similar results that
\begin{eqnarray}\label{inequality-decrease-i}
& & \LCal_\gamma\left(x_1^{k+1},\ldots, x_{i-1}^{k+1}, x_i^k, \ldots,\lambda^k\right) - \LCal_\gamma\left(x_1^{k+1},\ldots, x_i^{k+1}, x_{i+1}^k,\ldots,\lambda^k\right) \nonumber \\
&\geq& \frac{\gamma}{2} \left\| A_i x_i^k - A_i x_i^{k+1} \right\|^2.
\end{eqnarray}
Summing \eqref{inequality-decrease-1} and \eqref{inequality-decrease-i} over $i=2,\ldots,N$, we have
\begin{eqnarray}\label{L-sufficient-decrease-1}
& & \LCal_\gamma\left(x_1^k,\ldots,x_N^k,\lambda^k\right) -  \LCal_\gamma\left(x_1^{k+1},\ldots, x_N^{k+1},\lambda^k\right) \nonumber \\
&\geq& \frac{\gamma}{2}\sum\limits_{i=1}^{N-1}\left\| A_i x_i^k - A_i x_i^{k+1} \right\|^2 + \frac{\gamma}{2}\left\| x_N^k - x_N^{k+1} \right\|^2.
\end{eqnarray}
On the other hand, it follows from \eqref{opt-cond-N-1} that
\begin{eqnarray}\label{L-sufficient-decrease-2}
& & \LCal_\gamma\left(x_1^{k+1},\ldots, x_N^{k+1},\lambda^k\right) - \LCal_\gamma\left(x_1^{k+1},\ldots, x_N^{k+1},\lambda^{k+1}\right) \nonumber \\
& = & \frac{1}{\gamma}\left\| \lambda^k - \lambda^{k+1} \right\|^2 \geq - \frac{L^2}{\gamma}\left\| x_N^k - x_N^{k+1} \right\|^2.
\end{eqnarray}
Combining \eqref{L-sufficient-decrease-1} and \eqref{L-sufficient-decrease-2} yields
\begin{eqnarray}\label{inequality_decrease_total}
& & \LCal_\gamma\left(x_1^k,\ldots, x_N^k,\lambda^k\right) - \LCal_\gamma\left(x_1^{k+1},\ldots, x_N^{k+1},\lambda^{k+1}\right) \nonumber \\
& \geq & \frac{\gamma}{2}\sum\limits_{i=1}^{N-1}\left\| A_i x_i^k - A_i x_i^{k+1} \right\|^2 + \frac{\gamma^2 - 2L^2}{2\gamma} \left\| x_N^k - x_N^{k+1} \right\|^2 \nonumber \\
& = & \frac{\gamma}{2}\sum\limits_{i=1}^{N-1}\left\| A_i x_i^k - A_i x_i^{k+1} \right\|^2 + \frac{\gamma^2 - 2L^2}{2\gamma(1+L^2)}\left\| x_N^k - x_N^{k+1} \right\|^2 +  \frac{L^2(\gamma^2 - 2L^2)}{2\gamma(1+L^2)}\left\| x_N^k - x_N^{k+1} \right\|^2 \nonumber \\
& \geq & \frac{\gamma}{2}\sum\limits_{i=1}^{N-1}\left\| A_i x_i^k - A_i x_i^{k+1} \right\|^2 + \frac{\gamma^2 - 2L^2}{2\gamma(1+L^2)}\left( \left\| x_N^k - x_N^{k+1} \right\|^2 + \left\| \lambda^k - \lambda^{k+1} \right\|^2\right) \nonumber \\
& \geq & \frac{\gamma^2 - 2L^2}{2\gamma(1+L^2)}\left(\sum\limits_{i=1}^{N-1}\left\| A_i x_i^k - A_i x_i^{k+1} \right\|^2 + \left\| x_N^k - x_N^{k+1} \right\|^2 + \left\| \lambda^k - \lambda^{k+1} \right\|^2\right)
\end{eqnarray}
where the last inequality holds due to the fact that
\begin{displaymath}
\frac{\gamma}{2}\geq\frac{\gamma^2 - 2L^2}{2\gamma(1+L^2)}.
\end{displaymath}
\item Note that %We have the following series of inequalities
\begin{eqnarray*}
& & \LCal_\gamma\left(x_1^{k+1},\ldots, x_N^{k+1},\lambda^{k+1}\right) \\
& = & \sum_{i=1}^{N-1} f_i(x_i^{k+1}) + f_N(x_N^{k+1})- \left\langle \lambda^{k+1}, \sum_{i=1}^{N-1} A_i x_i^{k+1} + x_N^{k+1} -b\right\rangle + \frac{\gamma}{2}\left\|\sum_{i=1}^{N-1} A_i x_i^{k+1} + x_N^{k+1} -b\right\|^2.
\end{eqnarray*}
It follows from \eqref{opt-cond-N-1} and the fact that $\nabla f_N$ is Lipschitz continuous with constant $L$ that,
\begin{eqnarray*}
& & f_N\left( b - \sum\limits_{i=1}^{N-1} A_i x_i^{k+1}\right) \\
& \leq & f_N(x_N^{k+1}) + \left\langle \nabla f_N(x_N^{k+1}), \left( b - \sum\limits_{i=1}^{N-1} A_i x_i^{k+1} - x_N^{k+1}\right)\right\rangle + \frac{L}{2}\left\| b - \sum\limits_{i=1}^{N-1} A_i x_i^{k+1} - x_N^{k+1} \right\|^2 \\
& = & f_N(x_N^{k+1}) - \left\langle \nabla f_N(x_N^{k+1}), \sum\limits_{i=1}^{N-1} A_i x_i^{k+1} + x_N^{k+1} - b\right\rangle + \frac{L}{2}\left\| \sum\limits_{i=1}^{N-1} A_i x_i^{k+1} + x_N^{k+1} - b\right\|^2 \\
& = & f_N(x_N^{k+1}) - \left\langle \lambda^{k+1}, \sum\limits_{i=1}^{N-1} A_i x_i^{k+1} + x_N^{k+1} - b\right\rangle + \frac{L}{2} \left\| \sum\limits_{i=1}^{N-1} A_i x_i^{k+1} + x_N^{k+1} - b\right\|^2.
\end{eqnarray*}
This implies that there exists $L^*>-\infty$, such that
\begin{eqnarray}\label{L-lower-bound}
& & \LCal_\gamma(x_1^{k+1},\ldots, x_N^{k+1},\lambda^{k+1}) \nonumber \\
& \geq & \sum_{i=1}^{N-1} f_i(x_i^{k+1}) + f_N\left( b - \sum\limits_{i=1}^{N-1} A_i x_i^{k+1}\right) + \frac{\gamma-L}{2} \left\|\sum_{i=1}^{N-1} A_i x_i^{k+1} + x_N^{k+1} -b\right\|^2 \nonumber \\
& > & L^* ,
\end{eqnarray}
where the last inequality holds since $\gamma>L$ and $\inf_{\XCal_i}f_i>f_i^*$ for $i=1,2,\ldots,N$.

Therefore, it directly follows from \eqref{L-sufficient-decrease} and $\gamma>\sqrt{2}L$ that,
\begin{eqnarray*}
\frac{\gamma^2-2L^2}{2\gamma(1+L^2)}\sum\limits_{k=0}^K \left(\sum\limits_{i=1}^{N-1} \| A_i x_i^{k+1} - A_i x_i^k \|^2 + \| x_N^{k+1} - x_N^k \|^2 + \|\lambda^{k+1}-\lambda^k\|^2\right) \leq \LCal_\gamma(w^0) - L^*.
\end{eqnarray*}
Letting $K\rightarrow\infty$, we have
\begin{eqnarray*}
\frac{\gamma^2-2L^2}{2\gamma(1+L^2)}\sum\limits_{k=0}^\infty \left(\sum\limits_{i=1}^{N-1} \| A_i x_i^{k+1} - A_i x_i^k \|^2 + \| x_N^{k+1} - x_N^k \|^2 + \|\lambda^{k+1}-\lambda^k\|^2\right) \leq \LCal_\gamma(w^0) - L^*,
\end{eqnarray*}
which implies \eqref{inequality-x-lambda} and \eqref{inequality-x-lambda-limit}.

It also follows from \eqref{L-lower-bound}, \eqref{L-sufficient-decrease} and $\gamma>\sqrt{2}L$ that $\LCal_\gamma(w^0) - f_N^* \geq\sum_{i=1}^{N-1} f_i(x_i^{k+1})$. This implies that $\left\{\left(x_1^k, x_2^k, \ldots,x_{N-1}^k\right): k=0,1,\ldots\right\}$ is a bounded sequence by using the coerciveness of $f_i+\textbf{1}_{\XCal_i}, i=1,2,\ldots,N-1$. The boundedness of $\left(x_N^k, \lambda^k\right)$ can be obtained by using \eqref{scenario-2-update-lambda}, \eqref{lambda-bound-x} and \eqref{inequality-x-lambda-limit}.

\item From the definition of $\LCal_\gamma$, it is clear that for $i=1,\ldots,N-1$,
\begin{displaymath}
g_i(x_i^{k+1}) - A_i^\top\lambda^{k+1} + \gamma A_i^\top \left(\sum\limits_{i=1}^{N-1} A_i x_i^{k+1} + x_N^{k+1} - b\right) \in \partial_{x_i} \LCal_\gamma(w^{k+1}),
\end{displaymath}
and
\begin{displaymath}
\nabla f(x_N^{k+1}) - \lambda^{k+1} + \gamma \left(\sum\limits_{i=1}^{N-1} A_i x_i^{k+1} + x_N^{k+1} - b\right) = \nabla_{x_N} \LCal_\gamma(w^{k+1}),
\end{displaymath}
and
\begin{displaymath}
b - \sum\limits_{i=1}^{N-1} A_i x_i^{k+1} - x_N^{k+1} = \nabla_{\lambda} \LCal_\gamma(w^{k+1}),
\end{displaymath}
where $g_i\in\partial\left(f_i + \mathbf{1}_{\XCal_i}\right)$ for $i=1,2,\ldots,N-1$. Since \eqref{scenario-2-opt-x-1-lambda}, \eqref{scenario-2-opt-x-i-lambda}, and \eqref{scenario-2-opt-x-N-lambda} imply that
\begin{align}
& g_{1}(x_{1}^{k+1})-A_{1}^{\top}\lambda^{k+1} = - \gamma A_{1}^{\top}\left(\sum\limits_{j=2}^{N-1} A_{j}(x_{j}^{k}-x_{j}^{k+1})+(x_{N}^{k}-x_{N}^{k+1})\right), & \label{opt-x1-lambda-1-kl} \\
& g_i(x_{i}^{k+1})-A_{i}^{\top}\lambda^{k+1}= -\gamma A_{i}^{\top}\left(\sum\limits_{j=i+1}^{N-1} A_{j}(x_{j}^{k}-x_{j}^{k+1})+(x_{N}^{k}-x_{N}^{k+1})\right), & \label{opt-xi-lambda-1-kl} \\
& \nabla f_N(x_{N}^{k+1})-\lambda^{k+1} = 0, &\label{opt-xN-lambda-1-kl}
\end{align}
we have
\begin{align}
& R_i^{k+1} = \gamma A_i^\top \left(\sum\limits_{i=1}^{N-1} A_i x_i^{k+1} + x_N^{k+1} - b\right) - \gamma A_{i}^{\top}\left(\sum\limits_{j=i+1}^{N-1} A_{j}(x_{j}^{k}-x_{j}^{k+1})+(x_{N}^{k}-x_{N}^{k+1})\right) \in \partial_{x_i} \LCal_\gamma(w^{k+1}) , & \nonumber \\
& R_N^{k+1} = \gamma \left(\sum\limits_{i=1}^{N-1} A_i x_i^{k+1} + x_N^{k+1} - b\right) = \nabla_{x_N} \LCal_\gamma(w^{k+1}), &\nonumber \\
& R_{\lambda}^{k+1} = b - \sum\limits_{i=1}^{N-1} A_i x_i^{k+1} - x_N^{k+1} = \nabla_{\lambda} \LCal_\gamma(w^{k+1}), &\nonumber
\end{align}
for $i=1,2,\ldots,N-1$. This implies that $\left( R_1^{k+1}, \ldots, R_N^{k+1}, R_\lambda^{k+1}\right)\in\partial\LCal_\gamma(w^{k+1})$.

We now need to estimate the norms of $R_i^{k+1},1\leq i\leq N-1$ and $R_N^k$ and $R_\lambda^k$. It holds true that,
\begin{eqnarray*}
\left\| R_i^{k+1} \right\| & \leq & \gamma\left\| A_i^\top\right\| \left(\sum\limits_{j=i+1}^{N-1} \left\| A_j x_j^{k} - A_j x_j^{k+1}\right\| + \left\| x_N^{k} - x_N^{k+1}\right\| \right) + \gamma\left\| A_i^\top\right\|\left\| \sum\limits_{i=1}^{N-1} A_i x_i^{k+1} + x_N^{k+1} - b\right\|  \\
& \leq & \gamma\left\| A_i^\top\right\|\left(\sum\limits_{j=1}^{N-1} \left\| A_j x_j^{k} - A_j x_j^{k+1}\right\| + \left\| x_N^{k} - x_N^{k+1}\right\| \right) + \left\| A_i^\top\right\|\left\| \lambda^k - \lambda^{k+1}\right\|
\end{eqnarray*}
and
\begin{eqnarray*}
\left\| R_N^{k+1} \right\| \leq \gamma\left\| \sum\limits_{i=1}^{N-1} A_i x_i^{k+1} + x_N^{k+1} - b\right\| = \left\| \lambda^k - \lambda^{k+1}\right\| ,
\end{eqnarray*}
and
\begin{eqnarray*}
\left\| R_\lambda^{k+1} \right\| & = & \left\| \sum\limits_{i=1}^{N-1} A_i x_i^{k+1} + x_N^{k+1}- b\right\| = \frac{1}{\gamma}\left\| \lambda^k - \lambda^{k+1}\right\| .
\end{eqnarray*}
Therefore, we arrive at \eqref{inequality-subgradient-upper-bound}
where $M$ is defined in \eqref{def-M}.
\end{enumerate}
%\end{proof}

\hfill $\Box$

\bigskip

\textbf{Proof of Theorem \ref{thm-l2-convergence}.}
\begin{enumerate}
\item It has been proven in Lemma \ref{LemmaA1} that $\left\{\left(x_1^k, x_2^k, \ldots,x_N^k, \lambda^k\right): k=0,1,\ldots\right\}$ is a bounded sequence. Therefore, we conclude that $\Omega(w^0)$ is non-empty by the Bolzano-Weierstrass Theorem. Let $w^* = \left(x_1^*,\ldots, x_N^*,\lambda^*\right)\in\Omega(w^0)$ be a limit point of $\{w^k = \left(x_1^k,\ldots, x_N^k,\lambda^k\right):k=0,1,\ldots\}$. Then there exists a subsequence $\left\{w^{k_q} = \left(x_1^{k_q},\ldots, x_N^{k_q},\lambda^{k_q}\right):q=0,1,\ldots\right\}$ such that $w^{k_q}\rightarrow w^*$ as $q\rightarrow\infty$. Since $f_i, i=1,\ldots,N-1$, are lower semi-continuous, we obtain that
\begin{eqnarray}\label{lower_bound_f_i}
\liminf\limits_{q\rightarrow\infty} f_i(x_i^{k_q}) \geq f_i(x_i^*), \quad i=1,2,\ldots, N.
\end{eqnarray}
From the iterative step \eqref{scenario-2-update-x-1}-\eqref{scenario-2-update-lambda}, we have for any integer $k$ and any $i=1,\ldots,N-1$,
\begin{displaymath}
x_i^{k+1} := \argmin\limits_{x_i\in\XCal_i} \ \LCal_\gamma(x_1^{k+1},\ldots, x_{i-1}^{k+1}, x_i, x_{i+1}^k,\ldots, x_N^k;\lambda^k).
\end{displaymath}
Letting $x_i = x_i^*$ in the above, we get
\begin{eqnarray*}
\LCal_\gamma(x_1^{k+1},\ldots, x_i^{k+1}, x_{i+1}^k,\ldots, x_N^k;\lambda^k) \leq \LCal_\gamma(x_1^{k+1},\ldots, x_{i-1}^{k+1}, x_i^{*}, x_{i+1}^k,\ldots, x_N^k;\lambda^k),
\end{eqnarray*}
i.e.,
\begin{eqnarray*}
& & f_i(x_i^{k+1}) - \left\langle\lambda^k, A_i x_i^{k+1}\right\rangle + \frac{\gamma}{2}\left\| \sum\limits_{j=1}^i A_j x_j^{k+1} + \sum\limits_{j=i+1}^{N-1} A_j x_j^{k} + x_N^k - b \right\|^2 \\
& \leq &  f_i(x_i^*) - \left\langle\lambda^k, A_i x_i^*\right\rangle + \frac{\gamma}{2}\left\| \sum\limits_{j=1}^{i-1} A_j x_j^{k+1} + A_i x_i^* + \sum\limits_{j=i+1}^{N-1} A_j x_j^{k} + x_N^k - b \right\|^2.
\end{eqnarray*}
Choosing $k=k_q-1$ in the above inequality and letting $q$ go to $+\infty$, we obtain
\begin{eqnarray}\label{upper_bound_f_i}
\limsup\limits_{q\rightarrow +\infty}f_i(x_i^{k_q}) \leq \limsup\limits_{q\rightarrow +\infty} \left( \frac{\gamma}{2}\left\| A_i x_i^{k_q} - A_i x_i^* \right\|^2 - \left\langle\lambda^k, A_i x_i^{k_q}- A_i x_i^*\right\rangle \right) + f_i(x_i^*),
\end{eqnarray}
for $i=1,2,\ldots,N-1$. Here we have used the facts that both the sequence $\{w^k:k=0,1,\ldots\}$ is bounded, and $\gamma$ is finite, and that the distance between two successive iterates tends to zero \eqref{inequality-x-lambda-limit}, and the fact that
\begin{displaymath}
\sum\limits_{j=1}^i A_j x_j^{k+1} + \sum\limits_{j=i+1}^{N-1} A_j x_j^k + x_N^k - b  =  \sum\limits_{j=i+1}^{N-1} \left( A_j x_j^k - A_j x_j^{k+1}\right) + \left( x_N^k - x_N^{k+1}\right) + \frac{1}{\gamma} (\lambda^k - \lambda^{k+1}).
\end{displaymath}
From \eqref{inequality-x-lambda-limit} we also have $x_i^{k_q-1}\rightarrow x_i^*$ as $q\rightarrow\infty$, hence \eqref{upper_bound_f_i} reduces to $$\limsup\limits_{q\rightarrow \infty}f_i(x_i^{k_q})\leq f_i(x_i^*).$$ Therefore, combining with %in view of
\eqref{lower_bound_f_i}, $f_i(x_i^{k_q})$ tends to $f_i(x_i^*)$ as $q\rightarrow \infty$. Therefore, we can conclude that
\begin{eqnarray*}
\lim\limits_{q\rightarrow\infty}\LCal_\gamma(w^{k_q}) & = & \lim\limits_{q\rightarrow\infty}\left( \sum\limits_{i=1}^N f_i(x_i^{k_q}) - \left\langle\lambda^{k_q}, \sum\limits_{i=1}^{N-1} A_i x_i^{k_q} + x_N^{k_q} -b\right\rangle + \frac{\gamma}{2}\left\| \sum\limits_{i=1}^{N-1} A_i x_i^{k_q} + x_N^{k_q} -b \right\|^2\right) \\
& = & \sum\limits_{i=1}^N f_i(x_i^{*}) - \left\langle\lambda^{*}, \sum\limits_{i=1}^{N-1} A_i x_i^{*}+x_N^{*}-b\right\rangle + \frac{\gamma}{2}\left\| \sum\limits_{i=1}^{N-1} A_i x_i^{*}+x_N^{*}-b \right\|^2 \\
& = & \LCal_\gamma(w^{*}).
\end{eqnarray*}
On the other hand, it follows from \eqref{inequality-x-lambda-limit} and \eqref{inequality-subgradient-upper-bound} that
\begin{eqnarray}
\left( R_1^{k+1}, \ldots, R_N^{k+1}, R_\lambda^{k+1}\right) & \in & \partial\LCal_\gamma(w^{k+1}) \\
\left( R_1^{k+1}, \ldots, R_N^{k+1}, R_\lambda^{k+1}\right) & \rightarrow & (0,\ldots,0), \quad k\rightarrow\infty.
\end{eqnarray}
It implies that $(0,\ldots,0)\in\partial\LCal_\gamma(x_1^*,\ldots,x_N^*,\lambda^*)$ due to the closeness of $\partial\LCal_\gamma$. Therefore, $w^* = \left( x_1^*,\ldots,x_N^*,\lambda^*\right)$ is a critical point of $\LCal_\gamma(x_1,\ldots,x_N,\lambda)$.
\item The proof for this assertion directly follows from Lemma 5 and Remark 5 of \cite{Bolte-Sabach-Teboulle-2014}. We omit the proof here for succinctness.
\item We define that $L^*$ is the finite limit of $\LCal_\gamma(x_1^k,\ldots,x_N^k,\lambda^k)$ as $k$ goes to infinity, i.e.,
\begin{displaymath}
L^* = \lim\limits_{k\rightarrow\infty} \LCal_\gamma(x_1^k,\ldots,x_N^k,\lambda^k) .
\end{displaymath}
Take $w^*\in\Omega(w^0)$. There exists a subsequence $w^{k_q}$ converging to $w^*$ as $q$ goes to infinity. Since we have proven that
\begin{eqnarray*}
\lim\limits_{q\rightarrow\infty}\LCal_\gamma(w^{k_q}) = \LCal_\gamma(w^{*}),
\end{eqnarray*}
and $\LCal_\gamma(w^{k})$ is a non-increasing sequence, we conclude that $\LCal_\gamma(w^{*}) = L^*$, hence the restriction of $\LCal_\gamma(x_1,\ldots,x_N,\lambda)$ to $\Omega(w^0)$ equals $L^*$.
\end{enumerate}
\end{document}